\documentclass[reqno,11pt]{article} 

\usepackage[final]{pdfpages} 
\usepackage{url,relsize,mathtools} 
\usepackage{graphicx} 
\usepackage{subcaption} 
\usepackage{amsmath,amsthm,amsfonts,amssymb,amscd} 
\usepackage{tikz}
\usepackage{rotating}
\usepackage[export]{adjustbox}
\usepackage{tabularx}
\usepackage{datetime}

\usepackage{caption}
\usepackage[colorlinks=true]{hyperref}

\newcommand{\undersetbrace}[2]{\ensuremath{\underset{\mathlarger{#1}}{\mathsmaller{\underbrace{#2}}}}}

\newcommand{\figureref}[1]{Figure \ref{#1} (page \pageref{#1})} 

\DeclareMathOperator{\LHS}{LHS} 
\DeclareMathOperator{\RHS}{RHS} 
\DeclareMathOperator{\ErrorTerms}{ErrorTerms}

\DeclareMathOperator{\rot}{rot} 
\DeclareMathOperator{\SO}{SO} 
\DeclareMathOperator{\AmmannChair}{AmmannChair} 
\DeclareMathOperator{\TubingenTriangle}{TubingenTriangle} 
\DeclareMathOperator{\Ulam}{Ulam}

\newcommand{\bb}{\mathbb}
\newcommand{\R}{\bb R}

\newcommand{\TilingPoint}[2]{\begin{bmatrix}#1 \\ #2\end{bmatrix}} 

\tikzset{line/.style={thick,black}}
\tikzset{linered/.style={thick,red}}

\definecolor{tan}{RGB}{203,174,84}
\definecolor{darktan}{RGB}{148,124,53}
\definecolor{tanbrown}{RGB}{117,98,42}
\definecolor{chairblue}{RGB}{43,136,231}
\definecolor{chairdarkblue}{RGB}{16,81,148}
\definecolor{chairlightblue}{RGB}{122,174,227}

\numberwithin{figure}{section}

\theoremstyle{plain} 
\newtheorem{theorem}{Theorem}
\newtheorem{prop}[theorem]{Proposition}
\newtheorem{lemma}[theorem]{Lemma}

\numberwithin{theorem}{section}

\theoremstyle{remark} 
\theoremstyle{definition} 
\newtheorem{example}[theorem]{Example}
\newtheorem{remark}[theorem]{Remark}

\theoremstyle{definition}

\newtheorem*{theorem*}{Theorem}

\allowdisplaybreaks

\begin{document}

\title{Pair Correlation and Gap Distributions for Substitution Tilings and
       Generalized Ulam Sets in the Plane} 
\author{
        Maxie D. Schmidt \\ \href{mailto:maxieds@gmail.com}{maxieds@gmail.com} \\ 
        \href{mailto:mschmidt34@gatech.edu}{mschmidt34@gatech.edu} \\ 
        Georgia Institute of Technology \\ School of Mathematics \\ 
        Atlanta, GA 30332
       } 
\date{}

\maketitle 

\begin{abstract} 
We study empirical statistical and gap distributions of several important 
tilings of the plane. In particular, we consider the slope distributions, 
the angle distributions, pair correlation, squared-distance pair correlation, 
angle gap distributions, and slope gap distributions for the 
Ammann Chair tiling, the recently discovered fifteenth pentagonal tiling, and 
a few pertinent tilings related to these famous examples. 
We also consider the spatial statistics of generalized Ulam sets in two dimensions. 
Additionally, we carefully prove a tight asymptotic formula for the time steps in which 
Ulam set points at certain prescribed geometric positions in their plots in the plane 
formally enter the recursively-defined sets. 

The software we have developed to these generate numerical approximations to the 
distributions for the tilings we consider here is 
written in Python under the Sage environment and is released as
open-source software which is available freely on our websites. 
In addition to the small subset of tilings and other point sets in the plane 
we study within the article, our 
program supports many other tiling variants and is easily extended for 
researchers to explore related tilings and iterative sets. 
\emph{MSC Subject Class (2010)}: 52C20; 06A99; 11B05; 62H11; 52C23. \\
\emph{Last Revised}: \today
\end{abstract} 

\section{Introduction} 

The study of the spatial statistics of various point sets in the plane, and 
in particular the use of these statistics as a comparison to 
random processes, is by now well-established 
(for example, see~\cite{Baake} and the references within). 
In this paper, we describe empirical statistics for various families of 
self-similar and pentagonal tilings, focusing in particular on two 
well-known (families of) tilings: special chair tilings, such as the 
Ammann chair, and the 
numbered Pentagonal tilings in 
Section \ref{Section_AmmannChair} and 
Section \ref{Section_pentagonal_tilings}, respectively. 
Additionally we consider generalized forms of Ulam sets in two dimensions, their spatial statistics, and formalize other properties 
related to these sets in Section \ref{Section_UlamSets}. 
We hope that our work and the software we provide inspires more theoretical work on the study of limiting distributions for these point sets, and in particular proofs of existence and formulas for limiting gap distributions and pair correlation functions for angles, slopes, and distances.

\subsection{A unified perspective on two-dimensional substitution tilings and Ulam sets} 

We begin by unifying perspectives on substitution tilings and two-dimensional Ulam sets, 
and their directional distribution statistics. In both situations we have a nested sequence of 
recursively-generated, or recursively-defined finite sets $A_n$, where nested means 
$$A_0 \subset A_1 \subset \cdots \subset A_{n-1} \subset A_n.$$ 
We are interested in the limiting distribution of directions and the pair 
correlation plots in the sets $A_n$ as $n \rightarrow \infty$. 
The next two subsections define the specific classes of substitution tilings and 
two-dimensional Ulam set variants we consider within this article. 
In Section \ref{subSection_Intro_DefsStats} below we define the specific forms of the 
directions and statistical plots which we compare in the next few sections. 

\subsubsection{Substitution tilings} 

We start with a polygon $P$ (with a vertex at the origin $0$) which as a 
decomposition into pieces $P_1, \ldots, P_k$, each of which is affinely similar 
to the original polygon $P$, that is, there are affine maps $g_i$ so that 
$g_i(P) = P_i$ and $$P = \bigcup_{i=1}^k P_i.$$ A map $g: \R^2 \rightarrow \R^2$ 
is affine if $$g(x) = hx + v,$$ where $h \in GL(2, \R)$ and $v \in \R^2$. $A_0$ 
is the set of vertices of $P$, $A_1$ the set of vertices of $P_1, \ldots P_k$. 
Now we can decompose $$P_i = \bigcup_{j=1}^k g_i(P_j),$$ the sets of vertices of 
$g_i(P_j), i, j = 1, \ldots, k$ forms $A_2$, and so on, at each stage decomposing 
each piece into smaller pieces. 
For example, if we take the square $[0,1]^2$ and its decomposition into quarters, 
we get finer and finer square grids, equivalently we could look at larger and 
larger subsets of the integer lattice. In this case, we know the distribution of the 
slope gaps of this complete ``tiling'' as Hall's distribution. 

\subsubsection{Ulam sets} 
\label{subsubSection_UlamSet_Constructions}

In this case, we start with $A_0$ given by a pair of vectors, 
$$A_0 = \{v_0, v_1\}, A_1 = \{v_0, v_1, v_0+v_1\}, $$ and to build 
$A_{n+1}$ from the set $A_n$ we add the shortest of the vectors by their Euclidean norms 
which can be written uniquely as the sum of two vectors in $A_n$. 
Note that this set will be in the wedge defined by $v_0$ and $v_1$. We let $A_{\infty} = \bigcup A_n$, and refer to this as the Ulam set associated to $v_0, v_1$. 
In our special cases where we have only two initial vectors, 
we have by Kravitz-Steinerberger~\cite[Theorem 1]{ULAMSEQ-ULAMSETS} that these Ulam sets
correspond to the following subset of the lattice points $\{mv_0 + nv_1: m, n \in \mathbb \mathbb Z \cup \{0\} \}$ in the  wedge: boundary points $v_0 + n v_1$ and $n v_0+v_1$ for natural numbers 
$n \geq 0$ and  inner points $m v_0 + n v_1$ for $m,n \geq 3$ both odd positive integers 
(see Section \ref{subSection_UlamSets_defs_and_GeomProps}). 

These points are filled in \emph{gradually} at different finite time steps: 
for $a, b \in \mathbb Z \cup \{0\}$ with $av_0 + bv_1 \in A_{\infty}$, let $K(a,b)$ denote the time of its appearance, $$K(a,b) = \min \{ n: av_0 + bv_1 \in A_n\}.$$
Consider the $n^{th}$ line segment $$L_n = \{ xv_0 + yv_1: x, y \geq 0, x+ y = n+1\}$$ 
consisting of points in the Ulam set $A_{\infty}$ on the segment between $nv_0+v_1$ and $v_0+nv_1$. 
In \S\ref{Section_UlamSets} we prove a refinement of~\cite[Theorem 1]{ULAMSEQ-ULAMSETS} in the context of our so-termed ``\emph{timing distributions}'' which we define precisely in 
Section \ref{Section_UlamSets}. 
In particular, we state and prove a more exact form of the next result which succinctly 
summarizes the statement of our key theorem proved in the sections below. 

\begin{theorem*}[Timing Distributions] 
\label{theorem:ulam} 
There are positive constants $C_1, C_2$ such that for any 
$av_0 + bv_1 \in A_{\infty} \cap L_n,$ we have that 
$$C_1 \cdot n^2 \le K(a,b) \le C_2 \cdot n^2.$$ 
\end{theorem*}

To our knowledge the considerations of such timing distributions in the context of generalized 
Ulam sets are new and have not been considered elsewhere or in the references. 
The generalizations of the one-dimensional Ulam sequence to sequences of vectors in the plane 
is a fairly recent construction made by Kravitz and Steinerberger in 2017. 
Even fully classifying the plot structure, or inclusion properties of the lattice spanned by the 
initial vectors for $k \geq 3$ specified initial conditions remains an open problem. 
Thus there is much left to explore about the properties and graphs of these generalized 
Ulam sets in the plane, and our new results on the timing distributions of these sets is only 
a start to many other properties which we can approach by geometric arguments concerning these sets. 

\subsection{Definitions and statistical plots considered within the article} 
\label{subSection_Intro_DefsStats}

\subsubsection{Directions} 

Given such a sequence of subsets of $\R^2$, we consider various statistics 
associated to the \emph{directions} of points in $A_n$. For both Ulam and 
substitution, note that the set of possible angles (and slopes) is determined 
by the original set (either the polygon or the initial two vectors). So the set of 
directions (either slopes or angles) is contained in some interval $[a, b]$ 
(where for slopes we could have $a = -\infty$ or $b= \infty$, though often in 
examples we can use symmetry to reduce our set, for example, with the square 
substitution/integer grid, we can look at things of slope between $0$ and $1$). 

\subsubsection{Gap distributions and spatial statistics} 

For $x \in \R^2$, let $x = (x^{(1)}, x^{(2)})$, and consider the slope and 
angle $$s(x) = \frac{x^{(2)}}{x^{(1)}}, \theta(x) = \arctan s(x).$$ 
Write $$A_n = \{x_1, \ldots x_N\}$$ where the resulting $$s_k = s(x_k), \theta_k = \theta(x_k)$$ 
are in increasing order. Of course $N = N_n$ depends on $n$. 
We are interested in (among other things) the limiting histogram distributions of the 
next sets and the following underlying questions about these distributions:
\begin{description}

\item[Equidistribution of slopes.] 

\noindent
For $[c, d] \subset [a,b]$, does 
$$ \lim_{n \rightarrow \infty} \frac{ \#\{1 \le k \le N_n: s_k \in [c,d]\}}{N_n} = \frac{d-c}{b-a}\text{?}$$

\item[Gap distributions.] 

\noindent 
Let $g_i = N_n(s_{i+1} - s_i)$ (slope gaps) or $g_i = N_n(\theta_{i+1} - \theta_i)$ (angle gaps). 
For $[c, d] \subset [0, \infty]$, does $$ \lim_{n \rightarrow \infty} \frac{ \#\{1 \le k \le N_n: g_k \in [c,d]\}}{N_n}$$ exist? 

\item[Pair correlation.] 

\noindent 
For $[c, d] \subset [0, \infty]$, does 
$$ \lim_{n \rightarrow \infty} \frac{ \#\{1 \le j < k \le N_n: N_n(x_k - x_j) \in [c,d]\}}{N_n}$$ exist? 

\end{description}

\subsubsection{Some notes on obtaining limiting distributions}
\label{special} 
Gap distributions and pair correlation functions have been widely studied for various families of subsets of the plane, often arising from connections with low-dimensional dynamical systems and number theory. These include lattices, affine lattices, sets of saddle connections on translation surfaces (often arising from billiards in polygons), and more recently, cut-and-project quasicrystals, see, for example~\cite{Athreya, AChaika, GAP-DIST-SLOPES-GOLDENL, ACheung, Elkies, MS, BCZ, BZsurvey, MSquasi, UyanikWork, Work}.
The methods of proof in these results vary, but a common thread in 
several is using dynamics and equidistribution of unipotent flows on an 
appropriate moduli space which parameterizes deformations of our point sets 
(see~\cite{Athreya} for a general theorem regarding this strategy). 
In particular, this strategy is particularly effective when the original set has a large linear symmetry group.

The equidistribution of a sequence of finite lists $F(k)$ (here, the slopes or angles of our point sets $A_k$) can be viewed as a first-order test for randomness. In general, one considers the \emph{normalized gap sets}
$$
G(k) : = \left\{ N_k\left(F_k^{(i+1)}- F_k^{(i)}\right): 0 \le i < N(k)\right\},
$$
and given $0 \le a < b  \le \infty$, we consider the behavior of the proportion of gaps between $a$ and $b$, given by
$$\lim_{k \rightarrow \infty} \frac{\left| G(k) \cap (a, b)\right|}{N_k}.$$
If the sequence $F(k)$ is \emph{truly random}, that is, given by $$F(k) = \{ X_{(0)} \le X_{(1)} \le \ldots \le X_{(k)}\},$$ where the $\{X_{(i)}\}$ are the order statistics generated by independent, identically distributed (i.i.d.) uniform $[0, 1)$ random variables $\{X_n\}_{n=0}^{\infty}$, the gap distribution has converges to a \emph{Poisson process} of intensity $1$. Precisely, for any $t>0$,
\begin{equation}\label{eq:pois}
\lim_{k \rightarrow \infty} \frac{\left| G(k) \cap (t, \infty)\right|}{N_k} = e^{-t}
\end{equation}
In many of our examples, the empirical limiting gap distribution appears to be \emph{not} Poissonian, thus giving some indication of the underlying `non-random' structure of the sets.

\subsection{New results and conjectures} 

\subsubsection{Conjectures and new observations} 

Our experimental project provides the results of numerical computations 
performed on the \emph{Sage Math Cloud} servers of the 
empirical distributions for the 
pair correlation, tiling point angles, tiling point slopes, and the 
corresponding gap distributions of the angles and slopes for more than $40$ 
well-known tilings of the plane \cite{PROJECT-WEBSITE,TILINGS-ENCYCLOPEDIA}. 

We predict that the individual angle and slope distributions equidistribute for 
sufficiently large $N$ (or equivalently points within some large radius $R < \infty$). 
We also pose several conjectures on the smoothness of the limiting 
distributions of the gap and pair correlation distributions of the chair and 
pentagon tilings we consider in 
Section \ref{Section_Chair-Related_Tilings} and in 
Section \ref{Section_pentagonal_tilings}. 
We expect that due to the inwards recursive nature (deflation of subtiles rather 
than the scaling of inflation) we use to generate successive plots 
of the chair tiling examples that the 
pair correlation plots should be center-heavy with a monotone decrease 
towards either extreme at the tails of the distribution. 
That the limiting distributions for the gap and pair correlation plots 
appear to converge to smooth, piecewise continuous curves for large $N$ 
hints at more of the underlying structure of these important and interesting 
tiling variants. 
Given the interest of the chair tilings with respect to the group actions of 
$\mathbb{Z}^2$ and $\mathbb{R}^2$ on these tiling spaces 
(see, for example, \cite{TBLCHAIR-GRPACTIONS,TOP-TILINGSPCS}), we also 
suggest the computations of the exact empirical distributions corresponding 
to these tilings as the subject of another computational research article. 

In Section \ref{subsubSection_Comp_AC_SConnGL_TT}, 
we also connect the similarities between the gap distributions for the 
\texttt{AmmannChair} tiling, the saddle connections on the golden L 
studied in \cite{GAP-DIST-SLOPES-GOLDENL}, and the 
\texttt{TubingenTriangle} tiling to the symmetry groups for these 
distinct tilings, which see are substantial for these infinite tiling sets. 
In particular, we make observations on the 
orbits of these tiling sets under the matrices from the 
\emph{Hecke $(2, 5, \infty)$ triangle group}, which from the proofs given 
in the references \cite{Athreya,GAP-DIST-SLOPES-GOLDENL} 
are immediately suggested by the similarities of the 
slope gap distributions for the these tilings. Here, we consider the 
saddle connections on the golden L to be a set of 
``tiling points`` within our software implementation, so we henceforth interchangeably 
refer to infinite set of these vectors as the \texttt{SaddleConnGoldenL} 
``tiling`` within the article.

\subsubsection{Theoretical results for pair correlation of substitution tilings} 

Most of the results on pair correlations and gap distributions within this article are purely numerical and 
computational in nature and are used to suggest underlying theoretical 
results on the statistical and spatial distributions for the tilings we 
consider within the article. 
The exception is a numerically-verified theoretical 
analysis of the pair correlation distributions for the special 
Ammann chair tiling given in Section \ref{Section_AmmannChair}. 
In this section we recursively compute the pair correlation between points 
in the tiling by distances between tilings points reflected in all four 
quadrants. The stacked probability distribution function 
subplots reflecting this recursion are 
shown in Figure \ref{figure_AmmannChair_StackedPCPlots}, 
which provides computational reassurance that our theoretical methods are 
correct. 
The recursive procedure we introduce by special case analysis in this section 
can be applied to the cases of other substitution tilings. 

\subsection{Tilings supported by our new software tools} 

Our software project websites are located online at the following links 
\cite{PROJECT-WEBSITE}: 
\begin{itemize} 
     \setlength{\itemsep}{-1mm} 
\item \textbf{Tiling Gap Distributions and Pair Correlation Project: } \\ 
\url{http://www.math.washington.edu/wxml/tilings/index.php} \\ 
\url{https://github.com/maxieds/WXMLTilingsHOWTO} 

\item \textbf{Ulam Sets GitHub Repository: } \\ 
\url{https://github.com/maxieds/Ulam-sets} 

\end{itemize}
It is an important, integral part of the new computational results and the 
new software tools we provide in our experimental mathematics project 
which provides the computed empirical distributions and a software API for 
over $40$ well-known tilings, most of which, 
with the exception of the pentagonal tilings we consider in 
Section \ref{Section_pentagonal_tilings}, 
are found in the \emph{Tilings Encyclopedia} 
\cite{TILINGS-ENCYCLOPEDIA} \footnote{ 
     Other supported substitution and pentagonal tiling variants available in our 
     software implementation include the following tilings: 
     \texttt{AmmannA3}, \texttt{AmmannA4}, \texttt{AmmannChair}, 
     \texttt{AmmannChair2}, \texttt{AmmannOctagon} (Ammann-Beenker), 
     \texttt{Armchair}, \texttt{Cesi}, \texttt{Chair3}, \texttt{Danzer7Fold}, 
     \texttt{DiamondTriangle}, \texttt{Domino}, \texttt{Domino-9Tile}, 
     \texttt{Equithirds}, \texttt{Fibonacci2D}, \texttt{GoldenTriangle}, 
     \texttt{IntegerLattice}, \texttt{MiniTangram}, \texttt{Octagonal1225}, 
     \texttt{PChairs}, several pentagon tilings, \texttt{Pentomino}, 
     \texttt{Pinwheel}, \texttt{SaddleConnGoldenL}, \texttt{SDHouse}, 
     \texttt{Sphinx}, \texttt{STPinwheel}, \texttt{T2000Triangle}, 
     \texttt{Tetris}, \texttt{Trihex}, \texttt{TriTriangle}, 
     \texttt{TubingenTriangle}, and the \texttt{WaltonChair}. 
}. 
As such, we encourage the reader to visit the site, explore the 
complete listings of results posted on our website, and most importantly to 
explore these new tiling plots by comparing individual distributions between 
related tilings using the form at the bottom of the webpage. 

The new open-source software tools and API we have developed is extendable and 
easily modified, which allows other researchers to explore these and other 
tilings by extending our work over several months with the 
University of Washington. 
Moreover, the gap and pair correlation plots currently generated by the 
stock distribution of our software are not rigid computational ends to the 
project: the numerical plots and tiling point quantities that can be 
computed with the tilings we implement here are limited only by the 
scope and ingenuity of the experimental mathematician exploring these 
tilings with our new tools. 
For example, the computations of the symmetry groups we mention in 
Section \ref{subsubSection_Comp_AC_SConnGL_TT} are 
performed using a \emph{Sage} script utilizing the tilings API we developed for the 
project. We have also employed our tilings API to compute other cases of 
higher-order joint slope gap distributions for special tilings of the plane. 
Given our interest in using the numerical results we compute here to hint at 
theoretical results for these tilings, we also suggest comparisons between 
substitution tilings with similarly-shaped tiles, such as those classes of 
tilings suggested in the comparison form section of our project website 
\cite{PROJECT-WEBSITE}, including categories of chair-like-tiling 
variants, pentagonal tilings, rectangular and domino-shaped tilings, and 
triangular-shaped tiling variants, among others. 

We chose to write our program in Python for clarity and its 
interoperability with the \emph{Sage} and the \emph{SageMathCloud} platforms. 
The implementation we provide in the Python source code for our tilings 
statistics application employs polygon representations of the 
tiles, from which we then extract the distinct individual points 
in the tiling. 
We also employ several existing open-source implementations of these 
tilings written in the \textit{Mathematica} language to supplement our 
development in Python 
\cite{APERIODIC-TILINGS-COMP,AMANN-CHAIR-WDEMO,PENTAGON-WDEMO,AMMANN-TILES-WDEMO}.
The \texttt{IntegerLattice} tilings is an implementation of the integer lattice 
points in the first quadrant satisfying $0 < y < x \leq R$ which we have 
used for testing and calibration of our numerical results. 
In particular, we know the expected empirical distributions for 
the slope gaps of the integer lattice as Hall's distribution. 
Similarly, we have employed the \texttt{SaddleConnGoldenL} tiling variant 
discussed above for testing and calibration of the accuracy of the numerical 
computations of the gap distributions in our software since we know the 
exact limiting distribution for the slope gaps of this tiling 
\cite{GAP-DIST-SLOPES-GOLDENL}. 

\section{Chair substitution tilings} 
\label{Section_Chair-Related_Tilings} 

\subsection{The Ammann (A2) chair tiling} 
\label{Section_AmmannChair} 

\subsubsection{Definition of the Ammann chair tiling} 

The Ammann chair (\texttt{AmmannChair}) tiling examples given in 
Figure \ref{figure_tiling_initial_La3Sa3_ScaledDims} 
show the initial (scaled) tiling dimensions, and the 
corresponding tiling point vertices for the first few substitution 
steps of this chair tiling procedure 
\cite{TILINGS-AND-PATTERNS,2D-GEOM-GOLDEN-SECTION,TILINGS-ENCYCLOPEDIA}. 
The recursive procedure used to generate these tilings leads to 
slightly more formalized definitions of the 
results shown in the histogram plots computed in the listings 
of the computational data compiled below in Section 
\ref{subSection_AmmannChair_StackedPlot_Figures}. 
In particular, for the $2 \times 2$ real matrices, 
$M_1, M_2 \in \left\{0, \pm \varphi^{-1/2}, \pm \varphi^{-1}\right\}_{2x2}$, 
and the column vectors, $T_1$ and $T_2$, defined by 
\begin{align} 
\label{eqn_M2T2_M1T1_defs} 
M_2 := \varphi^{-1} \times \begin{bmatrix} 1 & 0 \\ 0 & -1 \end{bmatrix}, 
T_2 := \varphi^{3} \times \TilingPoint{0}{1}, 
M_1 := \varphi^{-1/2} \times \begin{bmatrix} 0 & -1 \\ 1 & 0 \end{bmatrix}, 
T_1 := \varphi^{5/2} \times \TilingPoint{1}{0}, 
\end{align} 
these tiling point sets are generated by a 
substitution and tile inflation / deflation procedure 
defined recursively according to the next set definitions 
given in \eqref{eqn_Ln_init_rec_def}. 
\begin{align} 
\label{eqn_Ln_init_rec_def} 
L_{N} & = 
  \begin{cases} 
     \left\lbrace 
     M_1 \TilingPoint{x}{y} + T_1 : \TilingPoint{x}{y} \in L_{N-1} 
     \right\rbrace 
     \bigcup 
     \left\lbrace 
     M_2 \TilingPoint{x}{y} + T_2 : \TilingPoint{x}{y} \in L_{N-2} 
     \right\rbrace, & \text{ if $N \geq 2$; } \\ 
     \left\lbrace 
     \TilingPoint{0}{0}, 
     \TilingPoint{\varphi^{5/2}}{0}, 
     \TilingPoint{\varphi^{5/2}}{\varphi^{2}}, 
     \TilingPoint{\varphi^{3/2}}{\varphi^{2}}, 
     \TilingPoint{\varphi^{3/2}}{\varphi^{3}}, 
     \TilingPoint{0}{\varphi^{3}} 
     \right\rbrace \bigcup \left\lbrace 
     \TilingPoint{0}{\varphi}, 
     \TilingPoint{\varphi^{1/2}}{\varphi}, 
          \TilingPoint{\varphi^{1/2}}{\varphi^{2}} 
     \right\rbrace, & \text{ if $N = 1$; } \\ 
     \left\lbrace 
     \TilingPoint{0}{0}, 
     \TilingPoint{\varphi^{5/2}}{0}, 
     \TilingPoint{\varphi^{5/2}}{\varphi^{2}}, 
          \TilingPoint{\varphi^{3/2}}{\varphi^{2}}, 
     \TilingPoint{\varphi^{3/2}}{\varphi^{3}}, 
     \TilingPoint{0}{\varphi^{3}} 
     \right\rbrace, & \text{ if $N = 0$. } 
  \end{cases}
\end{align} 
We illustrate our numerical results in the next subsections which 
consider a recursive procedure for generating the plots of the 
pair correlation distributions of the Ammann chair, and also 
observations on the relations between the empirical 
distributions for the Ammann chair tilings and the corresponding 
statistics for the saddle connections on the golden L 
(\texttt{SaddleConnGoldenL}) and for the 
Tubingen triangle (\texttt{TubingenTriangle}) tiling. 

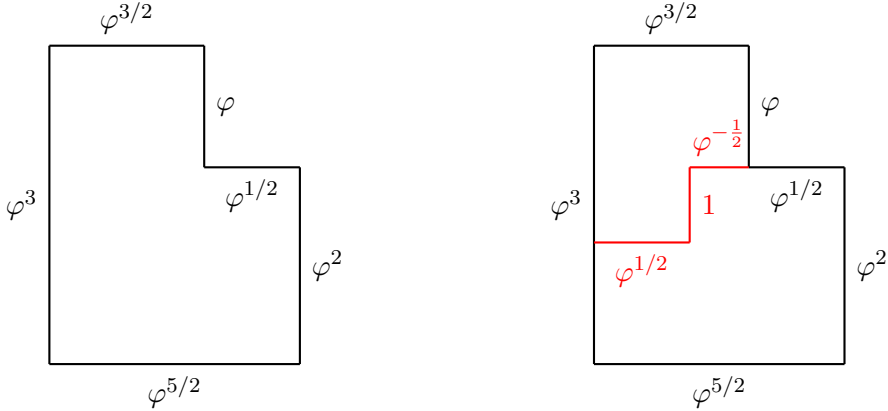
\begin{figure}[ht] 

\bigskip\hrule\hrule\bigskip 

\centering 
\begin{minipage}[b]{0.4\linewidth}
\begin{tikzpicture} 
\draw[line] (0,0) -- (3.33019, 0) node[midway, below = 0.5] {\(\varphi^{5/2}\)};
\draw[line] (3.33019, 0) -- (3.33019, 2.61803) node[midway,right = 0.5] {\(\varphi^{2}\)};
\draw[line] (3.33019, 2.61803) -- (2.05817, 2.61803) node[midway,below = 0.5] {\(\varphi^{1/2}\)};
\draw[line] (2.05817, 2.61803) -- (2.05817, 4.23606) node[midway,right = 0.5] {\(\varphi^{}\)};
\draw[line] (2.05817, 4.23606) -- (0, 4.23606) node[midway,above = 0.5] {\(\varphi^{3/2}\)};
\draw[line] (0, 4.23606) -- (0, 0) node[midway,left = 0.5] {\(\varphi^{3}\)};
\end{tikzpicture}
\end{minipage}
\begin{minipage}[b]{0.4\linewidth}
\begin{tikzpicture} 
\draw[line] (0,0) -- (3.33019, 0) node[midway, below = 0.5] {\(\varphi^{5/2}\)};
\draw[line] (3.33019, 0) -- (3.33019, 2.61803) node[midway,right = 0.5] {\(\varphi^{2}\)};
\draw[line] (3.33019, 2.61803) -- (2.05817, 2.61803) node[midway,below = 0.5] {\(\varphi^{1/2}\)};
\draw[line] (2.05817, 2.61803) -- (2.05817, 4.23606) node[midway,right = 0.5] {\(\varphi^{}\)};
\draw[line] (2.05817, 4.23606) -- (0, 4.23606) node[midway,above = 0.5] {\(\varphi^{3/2}\)};
\draw[line] (0, 4.23606) -- (0, 0) node[midway,left = 0.5] {\(\varphi^{3}\)};
\draw[linered] (0, 1.61803) -- (1.27202, 1.61803) node[midway,below = 0.5] {\(\varphi^{1/2}\)};
\draw[linered] (1.27202, 1.61803) -- (1.27202, 2.61803) node[midway,right = 0.5] {\(1\)};
\draw[linered] (1.27202, 2.61803) -- (2.05817, 2.61803) node[midway,above = 0.5] {\(\varphi^{-\frac{1}{2}}\)};
\end{tikzpicture}
\end{minipage}

\caption{Initial Tiling Dimensions and Deflation Rule for the Ammann Chair Tiling} 
\label{figure_tiling_initial_La3Sa3_ScaledDims}

\bigskip\hrule\hrule\bigskip 

\end{figure} 

\subsubsection{A special case theoretical analysis of the 
               pair correlation distributions 
               of the Ammann chair tiling} 
\label{subSection_AmmannChair_PC_recProc} 

First, we observe that the recursion implicit to 
\eqref{eqn_Ln_init_rec_def} provides a two--dimensional matrix procedure to 
perform the tiling substitution and inflation noted in the previous figures. 
We next iteratively apply this recursion to the pair correlation 
distance sets to see that we can approximate the full pair correlation 
plots, $\widetilde{D}_n$, by $\widetilde{D}_{n-1}$ and $\widetilde{D}_{n-2}$, 
as in Figure \ref{figure_tiling_initial_La3Sa3_ScaledDims}, and 
moreover, by the component--wise transformations of the 
distance sets corresponding to the next definitions of 
$\widetilde{D}_{i,m,n-k}$ for $i := 2,3,4,5$, $m := 0$, and 
$k \in [2, 5]$ in \eqref{eqn_D2345mn_defs}. 
In particular, the next sets correspond to the definitions of these 
transformation cases where the constant offsets, 
$\widetilde{\gamma}_{i,m,x}$ and $\widetilde{\gamma}_{i,m,y}$, 
correspond to predictable linear combinations of small powers of 
$\varphi := (1+\sqrt{5}) / 2$ and $\gamma := \sqrt{\varphi}$, and where the 
tiling lattice points, $(x_i, y_i) \in L_n$, 
are expressed as integer linear combinations of $\varphi$ depending on the 
recursion and the Fibonacci and Lucas numbers. 
\begin{align} 
\label{eqn_D2345mn_defs} 
\widetilde{D}_{2,m,n-j} & := \left\lbrace 
     \left(x_1+x_2 - \widetilde{\gamma}_{2,m,x}\right)^2 + 
     \left(y_1-y_2 - \widetilde{\gamma}_{2,m,y}\right)^2 : 
     \TilingPoint{x_1}{y_1}, \TilingPoint{x_2}{y_2} \in 
     \varphi^{-(4m+2)/2} \times L_{n-j} 
     \right\rbrace \\ 
\notag 
\widetilde{D}_{3,m,n-j} & := \left\lbrace 
     \left(x_1+x_2 - \widetilde{\gamma}_{3,m,x}\right)^2 + 
     \left(y_1+y_2 - \widetilde{\gamma}_{3,m,y}\right)^2 : 
     \TilingPoint{x_1}{y_1}, \TilingPoint{x_2}{y_2} \in 
     \varphi^{-(4m+3)/2} \times L_{n-j} 
     \right\rbrace \\ 
\notag 
\widetilde{D}_{4,m,n-j} & := \left\lbrace 
     \left(x_1-x_2 - \widetilde{\gamma}_{4,m,x}\right)^2 + 
     \left(y_1+y_2 - \widetilde{\gamma}_{4,m,y}\right)^2 : 
     \TilingPoint{x_1}{y_1}, \TilingPoint{x_2}{y_2} \in 
     \varphi^{-(4m+4)/2} \times L_{n-j} 
     \right\rbrace \\ 
\notag 
\widetilde{D}_{5,m,n-j} & := \left\lbrace 
     \left(x_1-x_2 - \widetilde{\gamma}_{5,m,x}\right)^2 + 
     \left(y_1-y_2 - \widetilde{\gamma}_{5,m,y}\right)^2 : 
     \TilingPoint{x_1}{y_1}, \TilingPoint{x_2}{y_2} \in 
     \varphi^{-(4m+5)/2} \times L_{n-j} 
     \right\rbrace 
\end{align} 
The constants implicit to the previous definitions in 
\eqref{eqn_D2345mn_defs} correspond to the next special cases in 
\eqref{eqn_SpecialCase_GammaXY_Constants}, 
though more general formulas for $\gamma_{j,m,xy}$ may be 
expanded in terms of powers of $\gamma = \sqrt{\varphi}$ and the 
Fibonacci numbers. 
\begin{align} 
\label{eqn_SpecialCase_GammaXY_Constants} 
\widetilde{\gamma}_{2,0,x}, \widetilde{\gamma}_{2,0,y} & := 
     \varphi^{-1} \times 2 \varphi^{3}, \varphi^{-1} \times 2 \varphi^{2} \\ 
\notag 
\widetilde{\gamma}_{3,0,x}, \widetilde{\gamma}_{2,0,y} & := 
     \varphi^{-3/2} \times 2 \varphi^{4}, 
     \varphi^{-3/2} \times 2 \varphi^{3} \\ 
\notag 
\widetilde{\gamma}_{4,0,x}, \widetilde{\gamma}_{2,0,y} & := 0, 0 \\ 
\notag 
\widetilde{\gamma}_{5,0,x}, \widetilde{\gamma}_{2,0,y} & := 
     \varphi^{-2} \times -2 \varphi^{3}, 
     \varphi^{-2} \times -2 \varphi^{2} 
\end{align} 
The computed plots shown in 
Figure \ref{figure_AmmannChair_StackedPCPlots} 
(page \pageref{figure_AmmannChair_StackedPCPlots}) 
then approximate the non-normalized, stacked histogram plots for the 
following sets defined in the notation of 
\eqref{eqn_D2345mn_defs} above as follows 
when ``$+$'' denotes set union: 
\begin{align} 
\label{eqn_AmannChairTiling_LHS_RHS_sets_approx_v1} 
\undersetbrace{\LHS(n)}{\widetilde{D}_n} & = 
     \undersetbrace{\RHS(n-1)}{(\varphi^{-1/2})^{2} \times \widetilde{D}_{n-1}} + 
     \undersetbrace{\RHS(n-2)}{(\varphi^{-1})^{2} \times \widetilde{D}_{n-2}} \\ 
\notag 
     & \phantom{=\quad} + 
     \undersetbrace{\RHS_2(n-2)}{\widetilde{D}_{2,0,n-2}} + 
     \undersetbrace{\RHS_3(n-3)}{\widetilde{D}_{3,0,n-3}} + 
     \undersetbrace{\RHS_4(n-4)}{\widetilde{D}_{4,0,n-4}} + 
     \undersetbrace{\RHS_5(n-5)}{\widetilde{D}_{5,0,n-5}} \\ 
\notag      
& \phantom{=\quad} + 
     \undersetbrace{\RHS(n-4)}{(\varphi^{-2})^{2} \times \widetilde{D}_{n-4}} 
     \pm \ErrorTerms_{n-7}. 
\end{align} 
Equivalently, we can define these stacked components of the pair correlation 
histogram plots show on the last pages in terms of transformations 
involving successive transformations of the $2 \times 2$ matrices defined in 
\eqref{eqn_M2T2_M1T1_defs}. 
Let the affine transformations, $L$ and $S$, of sets of tiling points in the 
plane be defined as 
\begin{align} 
\notag 
L \TilingPoint{x}{y} & := M_1 \TilingPoint{x}{y} + T_1 \\ 
\notag 
S \TilingPoint{x}{y} & := M_2 \TilingPoint{x}{y} + T_2, 
\end{align} 
so that for $N \geq 2$, the lattice point sets, $L_N$, 
in \eqref{eqn_Ln_init_rec_def} are alternately given by 
\begin{align} 
\notag 
L_n & = \left\{ 
     S \TilingPoint{x}{y} : \TilingPoint{x}{y} \in L_{N-2} 
     \right\} \bigcup \left\{      
     L \TilingPoint{x}{y} : \TilingPoint{x}{y} \in L_{N-1} 
     \right\}, 
\end{align} 
according to the substitution methods given in 
\eqref{eqn_Ln_init_rec_def} and shown in 
Figure \ref{figure_tiling_initial_La3Sa3_ScaledDims}. 
We may then write the previous approximation to the squared pair correlation distance sets 
from \eqref{eqn_AmannChairTiling_LHS_RHS_sets_approx_v1} as 
\begin{align} 
\label{eqn_AmannChairTiling_LHS_RHS_sets_approx_v2}
\widetilde{D}_n & \approx 
     \varphi^{-1} \widetilde{D}_{n-1} + \varphi^{-2} \widetilde{D}_{n-2} \\ 
\notag 
     & \phantom{\approx} + 
     \widetilde{D}\left(S_{n-2}, LL_{n-2}\right) + 
     \widetilde{D}\left(LS_{n-3}, SL_{n-3}\right) + 
     \widetilde{D}\left(SS_{n-4}, LSL_{n-4}\right) + 
     \widetilde{D}\left(LSS_{n-5}, SSL_{n-5}\right) \\ 
\notag 
     & \phantom{\approx} + 
     \widetilde{D}\left(SSS_{n-6}, SSL_{n-5}\right) + 
     \ErrorTerms_{2, n-7}, 
\end{align} 
where $\widetilde{D}(S_1, S_2)$ denotes the distance set between points in the 
two sets, $S_1$ and $S_2$, and where $M_N$ denotes the set of 
points, $M [x,y]^{T}$, for $[x,y]^{T} \in L_N$ and 
some $2 \times 2$ transformation matrix $M$, and where the operation, $SS$, 
corresponds to a single scaling and translation of points in the plane. 
The last observation suggests that we can iteratively apply this approximation 
to obtain further approximations to the distance sets, 
$\widetilde{D}\left(SSS_{n-6}, SSL_{n-5}\right)$. 

We expect that the limiting behavior for large $N$ of these pair correlation plots of the 
Euclidean distances between the tiling lattice points is approximately the 
stacked sum of skew normal pdfs for 
appropriately scaled bin sizes depending on $\varphi$ which 
tend to $0$ as $N \rightarrow \infty$. 
Additionally, expanding out the exact sums of Euclidean distances suggests 
convolutions with the exact lattice point distributions 
of the chair tiling. 
Leaving out the contributions of the intersection points in the 
approximations generated through this recursive procedure leads to 
error terms that are more pronounced towards the larger--distance 
tail of the exact distribution. 
We can perform a more careful analysis of these error terms 
corresponding to the intersection points between successive tiling 
steps. 

\subsubsection{A comparison of gap distributions and pair correlation 
               statistics between the Ammann chair tiling, 
               saddle connections on the Golden-L, and the 
               Tubingen triangle tiling} 
\label{subsubSection_Comp_AC_SConnGL_TT} 

The Golden-L surface considered by the special distributions in 
\cite{GAP-DIST-SLOPES-GOLDENL} 
corresponds exactly to the dimensions of the initial tile of the 
Ammann chair tiling shown in 
Figure \ref{figure_tiling_initial_La3Sa3_ScaledDims}. 
Similarly, the initial tile for the \texttt{TubingenTriangle} tiling 
corresponds to performing substitutions in the famous 
\emph{T\"ubingen triangle}, or scaled copies of the 
\emph{Robinson triangles} which bisect the \emph{golden triangle} in a 
natural way \cite{TUBINGEN-GAPS}. 
We have some intuition to believe that the empirical distributions we 
compute numerically here for the 
\texttt{AmmannChair}, \texttt{SaddleConnGoldenL}, and \texttt{TubingenTriangle} 
tilings supported by our software are related. 
In particular, 
we have an intuitive relationship between these three sets of tiling points 
which relates the symmetry groups of each respective tiling to their 
corresponding slope gap distributions. Namely, 
we know that the set of saddle connections on the golden L consists of 
two orbits (a so-termed ``\emph{long}`` and another ``\emph{short}`` 
subset orbit) of the 
Hecke $(2, 5, \infty)$ triangle group, which is generated by the two matrices 
$g_1$ and $g_2$ defined as follows: 
\begin{equation*} 
g_1 = \begin{bmatrix} 1 & \varphi \\ 0 & 1 \end{bmatrix} 
     \qquad \text{ and } \qquad 
g_2 = \begin{bmatrix} 1 & 0 \\ \varphi & 1 \end{bmatrix}. 
\end{equation*} 

\noindent
\textbf{The Ammann Chair Case}: 
In this case, it suffices to consider only the action of the two generator 
matrices, $g_1$ and $g_2$, on the \texttt{AmmannChair} tiling points. 
If we let $L_n = \AmmannChair(n)$ denote the 
set of tiling points in the \texttt{AmmannChair} 
tiling after $n$ substitution steps, we have 
evaluated the symmetry groups of these tiling points computationally to 
find that for each $i = 1, 2$ and for all $v \in L_n$ when $n \geq 1$, 
we can find a $m \geq 1$, a vector $w \in L_m$, and a (non-unique) 
real-valued scalar $c := a \cdot \varphi^{p/2}$ for some integers 
$a \geq 1$ and $p \geq 0$ such that $g_i v = c \cdot w$. \\ 

\noindent
\textbf{The T\"ubingen Triangle Case}: 
Similarly, in this case, we can consider only the action of the two generator 
matrices, $g_1$ and $g_2$, on the \texttt{Tubingen} tiling points. 
If we let $L_n = \TubingenTriangle(n)$ denote the 
set of tiling points in the \texttt{TubingenTriangle} 
tiling after $n$ substitution steps, we have also 
evaluated the symmetry groups of these tiling points computationally to 
find that for each $i = 1, 2$ and for all $v \in L_n$ when $n \geq 1$, 
we can find a $m \geq 1$, a vector $w \in L_m$, and a 
real-valued scalar $c$ such that $g_i v = c \cdot w$. 
We can say more about the action of each particular generator 
on the \texttt{TubingenTriangle} tilings points: 
for each $i = 1, 2$ and all vectors $v \in \TubingenTriangle(n)$, 
we can find real scalars, $c_i$, and corresponding coordinate stretching 
matrices, $M_{1,v} = \begin{bmatrix} 1 & 0 \\ 0 & c_1\end{bmatrix}$ and 
$M_{1,v} = \begin{bmatrix} c_2 & 0 \\ 0 & 1\end{bmatrix}$, such that 
$g_i v = M_{i,v} v$. \\ 

\noindent 
\textbf{Interpretations}: 
We provide a side-by-side comparison of the empirical 
slope gap distributions in 
\figureref{figure_EmpSlopeGapDist_SConnGoldenL_TT_tilings} 
and 
\figureref{figure_AmmannChair_EmpiricalSlopeGapDist}. 
There are strong characteristic 
similarities between the slope gap distribution for the saddle connections 
on the golden L in \cite{GAP-DIST-SLOPES-GOLDENL} and the 
slope gap distribution for the \texttt{TubingenTriangle} tiling which are 
noticed by inspection of the empirical gap distributions given in the 
first reference and in \cite{TUBINGEN-GAPS}. 
Conversely, this together with the proof given in the references 
\cite{Athreya,GAP-DIST-SLOPES-GOLDENL}
then also immediately suggests that the symmetry groups for the 
\texttt{SaddleConnGoldenL} and \texttt{TubingenTriangle} 
(and \texttt{AmmannChair}) tilings are related. 

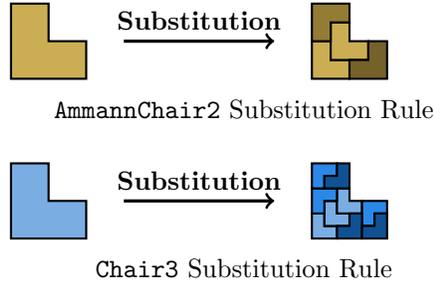
\begin{figure}[ht] 

\bigskip\hrule\hrule\bigskip 

\centering 
\begin{minipage}[b]{0.35\linewidth}
\begin{tikzpicture} 
\begin{scope}
\filldraw[thick,fill=tan,fill opacity=1] 
     (0, 0) -- (1, 0) -- (1, 0.5) -- (0.5, 0.5) -- 
     (0.5, 1) -- (0, 1) -- cycle;
\end{scope} 
\draw[very thick, ->] (1.5,0.5) -- node[above, text width=4cm, align=center]
    {\small\bfseries Substitution} (3.5,0.5);
\begin{scope}[shift={(4,0)}]
\filldraw[thick,fill=tan,fill opacity=1] 
     (0, 0) -- (0.5, 0) -- (0.5, 0.25) -- (0.25, 0.25) -- 
     (0.25, 0.5) -- (0, 0.5) -- cycle;
\filldraw[thick,fill=tan,fill opacity=1] 
     (0.25, 0.25) -- (0.75, 0.25) -- (0.75, 0.5) -- (0.5, 0.5) -- 
     (0.5, 0.75) -- (0.25, 0.75) -- cycle;
\filldraw[thick,fill=darktan,fill opacity=1] 
     (0, 0.5) -- (0.25, 0.5) -- (0.25, 0.75) -- (0.5, 0.75) -- 
     (0.5, 1) -- (0, 1) -- cycle;
\filldraw[thick,fill=tanbrown,fill opacity=1] 
     (0.5, 0) -- (1, 0) -- (1, 0.5) -- (0.75, 0.5) -- 
     (0.75, 0.25) -- (0.5, 0.25) -- cycle;
\end{scope} 
\end{tikzpicture}
\subcaption*{\texttt{AmmannChair2} Substitution Rule} 
\end{minipage}

\bigskip

\begin{minipage}[b]{0.35\linewidth}
\begin{tikzpicture} 
\begin{scope}
\filldraw[thick,fill=chairlightblue,fill opacity=1] 
     (0, 0) -- (1, 0) -- (1, 0.5) -- (0.5, 0.5) -- 
     (0.5, 1) -- (0, 1) -- cycle;
\end{scope} 
\draw[very thick, ->] (1.5,0.5) -- node[above, text width=4cm, align=center]
    {\small\bfseries Substitution} (3.5,0.5);
\begin{scope}[shift={(4,0)}]
\filldraw[thick,fill=chairlightblue,fill opacity=1] 
     (0, 0) -- (0.3333, 0) -- (0.3333, 0.1667) -- (0.1667, 0.1667) -- 
     (0.1667, 0.333) -- (0, 0.333) -- cycle;
\filldraw[thick,fill=chairlightblue,fill opacity=1] 
     (0.1667, 0.1667) -- (0.5, 0.1667) -- (0.5, 0.3333) -- (0.3333, 0.3333) -- 
     (0.3333, 0.5) -- (0.1667, 0.5) -- cycle;
\filldraw[thick,fill=chairlightblue,fill opacity=1] 
     (0.3333, 0.3333) -- (0.6667, 0.3333) -- (0.6667, 0.5) -- 
     (0.5, 0.5) -- (0.5, 0.6667) -- (0.3333, 0.6667) -- cycle;

\filldraw[thick,fill=chairblue,fill opacity=1] 
     (0, 0.3333) -- (0.1667, 0.3333) -- (0.1667, 0.5) -- 
     (0.3333, 0.5) -- (0.3333, 0.6667) -- (0, 0.6667) -- cycle;
\filldraw[thick,fill=chairblue,fill opacity=1] 
     (0, 0.6667) -- (0.1667, 0.6667) -- (0.1667, 0.8333) -- 
     (0.3333, 0.8333) -- (0.3333, 1) -- (0, 1) -- cycle;
\filldraw[thick,fill=chairblue,fill opacity=1] 
     (0.6667, 0.1667) -- (0.8333, 0.1667) -- (0.8333, 0.3333) -- 
     (1, 0.3333) -- (1, 0.5) -- (0.6667, 0.5) -- cycle;

\filldraw[thick,fill=chairdarkblue,fill opacity=1] 
     (0.1667, 0.6667) -- (0.5, 0.6667) -- (0.5, 1) -- 
     (0.3333, 1) -- (0.3333, 0.8333) -- (0.1667, 0.8333) -- cycle;
\filldraw[thick,fill=chairdarkblue,fill opacity=1] 
     (0.3333, 0) -- (0.6667, 0) -- (0.6667, 0.3333) -- 
     (0.5, 0.3333) -- (0.5, 0.1667) -- (0.3333, 0.1667) -- cycle;
\filldraw[thick,fill=chairdarkblue,fill opacity=1] 
     (0.6667, 0) -- (1, 0) -- (1, 0.3333) -- 
     (0.8333, 0.3333) -- (0.8333, 0.1667) -- (0.6667, 0.1667) -- cycle;
\end{scope} 
\end{tikzpicture}
\subcaption*{\texttt{Chair3} Substitution Rule} 
\end{minipage}

\captionof{figure}{Substitution Rules for the 4-Tile \texttt{AmmannChair2} and 
                   9-Tile \texttt{Chair3} Tiling Variations} 
\label{figure_49-Tile_ChairTiling_SubstRules}

\bigskip\hrule\hrule\bigskip 

\end{figure} 

\subsection{Statistical plots and gap distributions for L-shaped chair tilings} 
\label{subSection_AmmannChair_StackedPlot_Figures}

We next compare the empirical distributions of the non-periodic 
4-tile and 9-tile chair tilings in 
Figure \ref{figure_49-Tile_ChairTiling_SubstRules}. 
These examples are of interest primarily 
due to the group actions of $\mathbb{Z}^2$, $\mathbb{R}^2$, and a few 
other groups on their associated tiling spaces 
\cite{TBLCHAIR-GRPACTIONS,TOP-TILINGSPCS}. 
In particular, we can form three distinct 
tiling spaces out of each of these examples by letting $\Omega_1$ equal the 
set of all chair tilings with edges parallel to the coordinate axes, 
$\Omega_{\rot}$ equal the set of all chair tilings in any orientation 
with respect to the coordinate axes, and $\Omega_0$ denote the set of all 
chair tilings modulo rotation about the origin. 
The space $\Omega_1$ is then the closure of the \emph{translational} orbit 
of any one chair tiling, $\Omega_{\rot}$ is the closure of the Euclidean 
orbit of any chair tiling, and we have that 
$\Omega_0 = \Omega_1 / \mathbb{Z}_4 = \Omega_{\rot} / S^1$ (4-tile case) and 
$\Omega_0 = \Omega_1 / \mathbb{Z}_9 = \Omega_{\rot} / S^1$ (9-tile case) 
where the $\mathbb{Z}_m$ are finite rotation groups acting on $\Omega_1$ and 
$S^1 = \SO(2)$ is an infinite rotation group acting on $\Omega_{\rot}$ 
\cite[\S 4.5]{TOP-TILINGSPCS}. 

The stacked subplots shown in 
\figureref{figure_AmmannChair_StackedPCPlots} 
for the pair correlation plots corresponding to the 
recursively-defined functions in 
Section \ref{subSection_AmmannChair_PC_recProc} 
provide information at the relatively large substitution step sizes of 
$N := 13, 14$. 
Without the additional structure we develop in the previous section, 
we are only able to compare the Ammann Chair tilings at comparatively 
small step sizes of $N \leq 10$. 
The statistical plots corresponding to pair correlation, the 
gap distribution of the slopes, and the gap distribution of the angles 
we generate numerically also appear convergent to 
smooth curves for their respective limiting distributions when 
scaled by the number of points, $\binom{L_n}{2}$. 
We find, as predicted, that the distributions of the angles in the 
Ammann Chair tilings roughly equidistribute for large $N$, whereas the 
slope distributions for this tiling form a non-constant, approximately 
decreasing distribution concentrated near the origin. 

For comparison to the 
2-tile 
\texttt{AmmannChair} tiling which is partitioned by scaled copies of the 
Golden L at each substitution step, we also consider relations between the 
respective cases of the 4-tile \texttt{AmmannChair2} and the 
9-tile \texttt{AmmannChair3} 
chair tiling variants whose substitution rules are shown in 
Figure \ref{figure_49-Tile_ChairTiling_SubstRules}. 
The numerical approximations to the empirical distributions for each of these 
respective variations of the chair tilings when $N := 6$ and $N := 4$ 
is shown in \figureref{figure_ChairTilineVariants_DistsSummary}. 
Other chair-like tiling variants which we consider in our complete set 
of online results include the 
\texttt{AmmannA3}, \texttt{AmmannA4}, \texttt{Armchair}, 
\texttt{PChairs} (or the \emph{pregnant chairs} tiling), 
\texttt{Pentomino}, and the \texttt{WaltonChair} tilings. 
A comparison of the empirical distributions of these other notable 
chair tilings cases is given through the comparison form 
views available on our tilings project website \cite{PROJECT-WEBSITE}. 

\begin{figure}[h]

\bigskip\hrule\hrule\bigskip 

\begin{minipage}[b]{\linewidth}

\begin{minipage}{0.18\linewidth}
\centering\includegraphics[frame,scale=0.5]{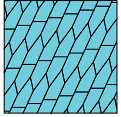} 
\caption*{\underline{\texttt{Pentagon1}}}
\end{minipage} 
\begin{minipage}{0.18\linewidth}
\centering\includegraphics[frame,scale=0.5]{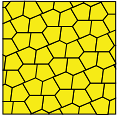} 
\caption*{\texttt{Pentagon2}}
\end{minipage} 
\begin{minipage}{0.18\linewidth}
\centering\includegraphics[frame,scale=0.5]{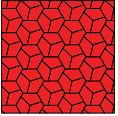} 
\caption*{\underline{\texttt{Pentagon3}}}
\end{minipage} 
\begin{minipage}{0.18\linewidth}
\centering\includegraphics[frame,scale=0.5]{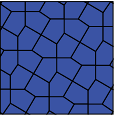} 
\caption*{\underline{\texttt{Pentagon4}}}
\end{minipage} 
\begin{minipage}{0.18\linewidth}
\centering\includegraphics[frame,scale=0.5]{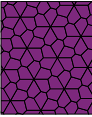} 
\caption*{\texttt{Pentagon5}}
\end{minipage} 

\end{minipage} 

\bigskip

\begin{minipage}[b]{\linewidth}

\begin{minipage}{0.23\linewidth}
\centering\includegraphics[frame,scale=0.5]{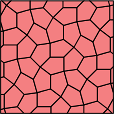} 
\caption*{\texttt{Pentagon8}}
\end{minipage} 
\begin{minipage}{0.23\linewidth}
\centering\includegraphics[frame,scale=0.5]{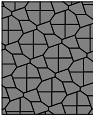} 
\caption*{\underline{\texttt{Pentagon10}}}
\end{minipage} 
\begin{minipage}{0.23\linewidth}
\centering\includegraphics[frame,scale=0.5]{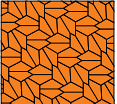} 
\caption*{\underline{\texttt{Pentagon11}}}
\end{minipage} 
\begin{minipage}{0.23\linewidth}
\centering\includegraphics[frame,scale=0.5]{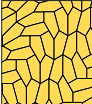} 
\caption*{\underline{\texttt{Pentagon15}}}
\end{minipage} 

\end{minipage} 

\bigskip\hrule\bigskip

\caption{\textbf{The Pentagonal Tiling Variations Supported by Our Software} \\ 
        {\small 
             The images of these tilings are taken from the interactive demo 
             provided by Ed Pegg \cite{PENTAGON-WDEMO}. 
             The underlined pentagonal tiling variants shown above for 
             special default settings of their implicit parameters are 
             featured in the empirical pair correlation 
             distributions shown in 
             \figureref{figure_PCDists_PentagonTilings_Comp}. 
         }
         }
\label{figure_pentagon_tiling_variants}

\bigskip\hrule\hrule\bigskip 

\end{figure} 

\section{Pentagonal tilings of the plane} 
\label{Section_pentagonal_tilings}

The constructions of the fifteen known distinct types of 
pentagonal tilings of the plane are one of the most famous problems in 
mathematics that has occupied mathematicians and housewives alike for 
most of the $20^{th}$ and $21^{st}$ centuries. 
One of Hilbert's twenty-three open problems in mathematics published in $1900$ 
considers finding equations for a full set of 
distinct classes of pentagonal tilings that cover the plane. 
Hilbert's $18^{th}$ problem is partially resolved by the fourteen 
types of pentagon tilings discovered by K.\ Reinhardt (1918), 
R.\ B.\ Kershner (1968), R.\ James (1975), M.\ Rice, and 
R.\ Stein (1985) throughout the twentieth century. 
The fifteenth known class of these famous pentagonal tilings 
was recently discovered by computer assisted researchers in $2015$. 
An exhaustive interactive demonstration of all fifteen of the pentagon tilings 
provided by the reference \cite{PENTAGON-WDEMO} 
shows the other variations including the 
nine tiling variants illustrated in 
Figure \ref{figure_pentagon_tiling_variants}. 

We have implemented a subset of these tilings in our software to 
perform statistics on the tile vertices of patches of the 
first, second, third, fourth, fifth, eighth, tenth, eleventh, and the 
most recent fifteenth pentagonal tilings in the forms of the 
default parameters to these tilings provided by the interactive 
demonstration mentioned above in 
Figure \ref{figure_pentagon_tiling_variants}. 
Notice that these pentagonal tiling variants 
comprise the examples of non-substitution tilings provided by 
default in our software. 
We also provide a \emph{Sage} worksheet for manipulating the statistical 
gap distribution plots of the newest fifteenth pentagonal tiling on our 
project website \cite{PROJECT-WEBSITE}. 

We analyze the plots of the 
pair correlation between these vertex sets in the comparisons 
given in the figures below. 
Side-by-side comparisons summarizing the empirical 
pair correlation between points in these tilings computed using our 
software is given in 
\figureref{figure_PCDists_PentagonTilings_Comp}. 
The open-source software we provide on our website 
is easily modified to support other variations of these nine 
parameterized classes of pentagon tilings. 
To our knowledge, the results in the figures below provide the first 
computations of the gap and pair correlation distributions for these 
pentagonal tilings, and in particular, for the newest fifteenth pentagonal 
tiling recently discovered within the last two years. 

\section{Ulam sets in two dimensions} 
\label{Section_UlamSets}

\subsection{Definitions of Ulam sets and their geometric properties} 
\label{subSection_UlamSets_defs_and_GeomProps}

We first recall the two-dimensional Ulam set construction introduced in \S\ref{subsubSection_UlamSet_Constructions}. 
The next definitions make our intuitive definition of two-dimensional Ulam sets from the introduction and their 
key geometric properties more precise. 

\subsubsection{Ulam sets} 
Given two linearly independent initial vectors $v_0, v_1 \in \mathbb{R}_{\geq 0}^2$, define the \emph{Ulam set} after $N := 0$ 
time steps to be $$U_0 := \{v_0, v_1\}.$$ We then construct subsequent stages of the Ulam set after 
$N \geq 1$ time steps recursively according to the formula 
\[U_N = U_{N-1} \bigcup \left\{v+w : v \neq w \in U_{N-1} \text{ and } ||v+w|| = M_{N-1}\right\},\] 
where 
\[M_N := \min \left\{||w_1+w_2|| : w_1, w_2 \in U_{N}, w_1 \neq w_2, 
     w_1+w_2 \neq w_3+w_4 \text{ for any two pairs } w_1 \neq w_2, w_3 \neq w_4\right\},\] 
and where $||w|| \equiv ||w||_2$ denotes the two-norm, or Euclidean distance of the vector $w$ from the 
origin\footnote{ 
     This is an important distinction to be made since if we use the supremum norm to measure the 
     magnitude of the vectors in our sets we obtain 
     drastically different, and much more structurally irregular, 
     Ulam set points \cite[cf.\ \S 3]{ULAMSEQ-ULAMSETS}. In particular, we are modeling our Ulam set 
     constructions after the examples in two-dimensions found in the reference by 
     Kravitz and Steinerberger which states a two-dimensional lattice theorem 
     governing the structure of our Ulam sets when we have two initial vectors which define these sets. 
}. 
That is, we define the Ulam set $U_N$ after $N$ timing steps to be the Ulam set $U_{N-1}$ 
together with the set of all vectors of minimal norm that can be uniquely written as the 
sum of two vectors in the previous set $U_{N-1}$. We let $U_{\infty} = \bigcup_{N \geq 0} U_N$. \\ 

\subsubsection{The structure of the Ulam set points} 
The infinite set 
$U_{\infty}$ consists of points of the form $av_0+bv_1$ where $(a,b) = (n, 1), (1, n)$ for some $n \in \mathbb{Z}_{\geq 0}$, or 
$(a, b) = (m, n)$ for some \emph{odd} positive integers $m,n \geq 3$ \cite[Thm.\ 1]{ULAMSEQ-ULAMSETS}. 
We also consider the sequence of $n^{th}$ line segments $L_n$ in the wedge spanned by $v_0$ and $v_1$ 
defined by $$L_n = \{xv_0 + yv_1: x, y \geq 0, x+y = n+1\}.$$
Geometrically, these $n^{th}$ line segments correspond to the negatively-sloped lines on the cross grains of the 
wedge defined by the vectors $(1, \varphi), (\varphi, 1)$ in 
Figure \ref{figure_proto_UlamSetPoints_v2}. 
The number of Ulam set points on the $n^{th}$ line segment $L_n$ is given by 
\[
\#_{\Ulam}(n) = 
     \begin{cases} 
     2, & \text{ if $n$ is even; } \\ 
     \frac{(n+1)}{2} & \text{ if $n$ is odd. }
     \end{cases} 
\]
As we will show in proving Theorem \ref{conj_UlamSetTimings_v1} in the next subsections, the 
$L_n$ form natural bins for the Ulam set points which are filled in approximately in 
order, as in $L_n$ before $L_{n+1}$, and so on. 
Thus we may construct so-termed ``\emph{timing distributions}'' based on the intervals of 
time steps at which these $n^{th}$ bounded segments are filled by the vectors in $U_N$. 

\subsubsection{Representations of the vectors on the $n^{th}$ line segments} 
One central point to understanding the proof of Theorem \ref{conj_UlamSetTimings_v1} 
given below is to note the several different, but equivalent representations of the 
vectors on the constructions of the $n^{th}$ geometric line segments which are defined above. 
In particular, when $n$ is \emph{even} the two vectors on $L_n$ are given by $nv_0+v_1$ and $v_0+nv_1$, and 
when $n$ is \emph{odd} we have the following three key representations of the 
vectors on $L_n$ which result by symmetry in the representations involving $v_0,v_1$: 
\begin{align*} 
\tag{I} 
& av_0+bv_1, && a,b \geq 1 \text{ both odd, } a+b = n+1; \\ 
\tag{II} 
& (n+1-b)v_0+bv_1, bv_0+(n+1-b)v_1, && \text{$b = 1, 3, 5, \ldots, n$; } \\ 
\tag{III} 
& \frac{n+1-2d}{2}(v_0+v_1) + 2dv_0, \frac{n+1-2d}{2}(v_0+v_1) + 2dv_1, && 
     \text{$d \in \left[-\frac{n-1}{2}, \frac{n-1}{2}\right]$, } 
     b \mapsto \frac{n+1}{2}-d. 
\end{align*} 

\subsubsection{Prototypical examples of the Ulam sets} 

We first consider the special cases of the Ulam set construction defined in 
the previous subsection corresponding to the 
initial vectors $\{v_0, v_1\} := \{(1, \varphi), (\varphi, 1)\}$ and one case of an 
Ulam set arising from two non-parallel initial vectors chosen randomly from $[0, 1]^2$. 
Figure \ref{figure_proto_UlamSetPoints_v2} and 
Figure \ref{figure_random_UlamSetPoints_v8} provide the respective points in these two special infinite Ulam sets 
truncated after $N := 250$ steps of the recursive procedure used to 
generate successive points in the sets. 
Plots of the empirical slope gap and pair correlation distributions for two Ulam sets defined by the 
initial vector cases of $\{v_0, v_1\} := \{(1, \varphi), (\varphi, 1)\}$ (as above) and the special case where 
$\{v_0, v_1\} := \{(1, 0), (0, 1)\}$ are compared in 
Figure \ref{figure_proto_UlamSetDists_v1} on page \pageref{figure_proto_UlamSetDists_v1}. 

%
%
%

\begin{figure}[ht!] 
\centering
\includegraphics[height=0.4\textheight,fbox]{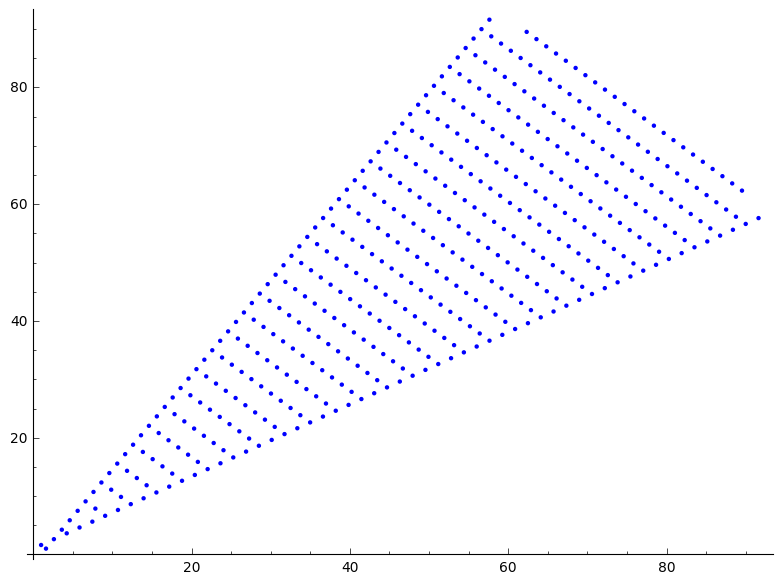} \\ 

\centering
\includegraphics[height=0.4\textheight,fbox]{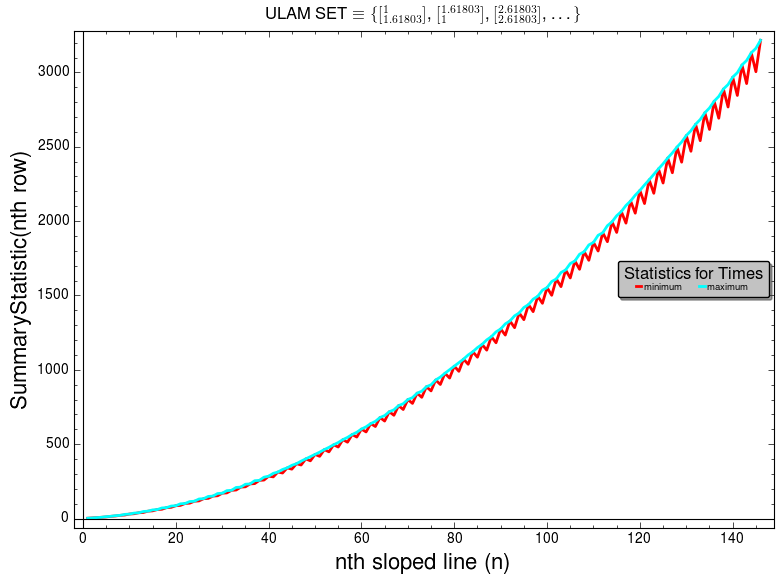}

\caption{The Ulam set arising from $\{(1, \varphi), (\varphi, 1)\}$ generated by running 
         \mbox{\texttt{\$ sage time-plots-random.sage userdef-2d 1 1.618034 1.618034 1 1500}} with our software. 
         The first image shows a plot of the Ulam set
         after $N := 250$ time steps. The second image shows the timing distribution of the 
         maximum and minimum entry times of the Ulam set points on the $n^{th}$ line segments plotted 
         over $n$ where the vectors involved in the timing plot correspond to the Ulam set after 
         $N := 1500$ steps. } 
\label{figure_proto_UlamSetPoints_v2}

\end{figure} 

\begin{figure}[ht!] 
\centering
\includegraphics[height=0.4\textheight,fbox]{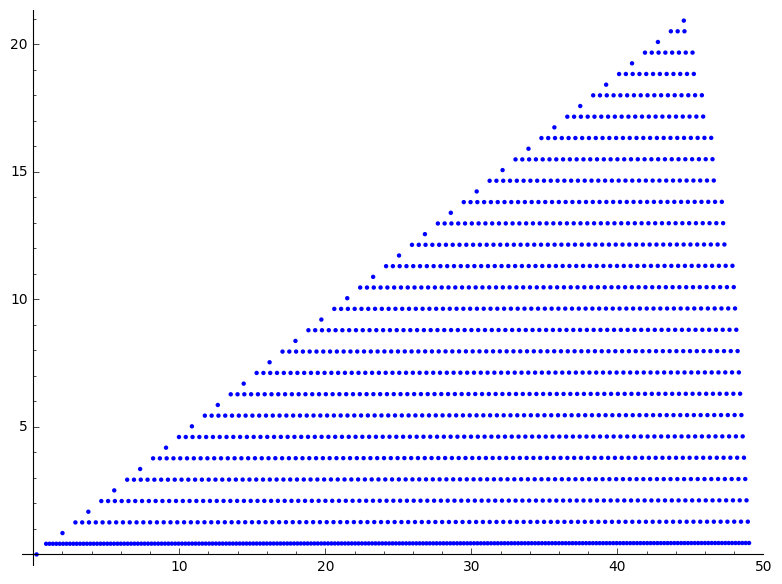} \\ 

\centering
\includegraphics[height=0.4\textheight,fbox]{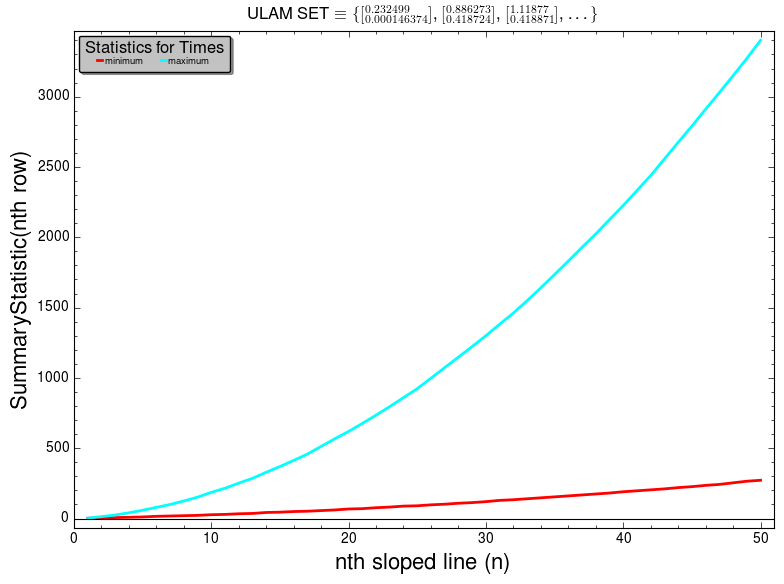}

\caption{An Ulam set with randomly generated initial vectors in $[0,1]^2$ generated by running 
         \mbox{\texttt{\$ sage time-plots-random.sage random-2d 1500}} with our software. 
         The first image shows a plot of the Ulam set
         after $N := 1500$ time steps. The second image shows the timing distribution of the 
         maximum and minimum entry times of the Ulam set points on the $n^{th}$ line segments plotted 
         over $n$ where the vectors involved in the timing plot correspond to the Ulam set after 
         $N := 1500$ steps. } 
\label{figure_random_UlamSetPoints_v8}

\end{figure} 

\subsection{Timing distributions of the generalized Ulam sets} 
\label{subSection_UlamSets_Timing} 

\subsubsection{An overview of timing properties in Ulam sets} 

We know both as a set definition and geometrically from 
the previous subsections what the points in the 
infinite Ulam sets correspond to as the number of steps $N \rightarrow \infty$ for some 
prescribed initial vectors, and moreover we see that these points are filled in 
gradually each over some fixed step $N < \infty$. 
It is then natural to ask questions about the timings of when Ulam set points at 
certain geometrical positions in the infinite set (or lattice points in the wedge defined by 
$v_0$ and $v_1$) formally enter the sets $U_N$. 
To the best of our knowledge such questions and 
conjectures relating to this subject matter have not yet been posed in the literature on 
generalized Ulam sets. 

Plots of a subset of the timing summary statistics we consider in our Ulam sets source code 
for the two prototypical cases from the previous subsection are also provided in the bottom rows of 
Figure \ref{figure_proto_UlamSetPoints_v2} and Figure \ref{figure_random_UlamSetPoints_v8} 
when the number of steps used to generated the Ulam set is $N := 1500$. 
The results and the proof of Theorem \ref{conj_UlamSetTimings_v1} given in the 
next subsection are essentially formulated in order 
to relate the number of steps $N$ to the vectors on each $L_n$, 
which are a priori distinct and unrelated parameters in the Ulam set construction. 

\subsubsection{A theorem on the timing distribution of minimum and maximum means} 

\begin{theorem}[The Timing Distribution of the Maximum and Minimum Means] 
\label{conj_UlamSetTimings_v1} 
If $T_{\max}(n)$ and $T_{\min}(n)$ denote the maximum and minimum entry times of the vectors on the 
$n^{th}$ line segment $L_n$, respectively, we have that 
there exist constants $C_1,C_2,C_3,C_4 > 0$ such that for all sufficiently large $n$ we have that 
\begin{align*} 
C_1 \cdot n^2 \leq T_{\max}(n) \leq C_2 \cdot n^2 \quad\text{ and }\quad 
     C_3 \cdot n^2 \leq T_{\min}(n) \leq C_4 \cdot n^2, 
\end{align*} 
i.e., that there is a large $n_0 \geq 1$ such that the inequalities above hold for all $n \geq n_0$. 
More precisely, if we define the maximum and minimum means of the entry times of the vectors on $L_n$ to be 
\begin{align*} 
T_{\max}(n) & = \max\ \left\{N \geq 1: \text{some vector $w \in L_n$ satisfies } 
                            w \in U_N \text{ and } w \notin U_{N-1}\right\} \\ 
T_{\min}(n) & = \min\ \left\{N \geq 1: \text{some vector $w \in L_n$ satisfies } 
                            w \in U_N \text{ and } w \notin U_{N-1}\right\}, 
\end{align*} 
then we have the tight asymptotic bounds on the growth of these two functions stated above: 
$T_{\max}(n) = \Theta(n^2)$ and $T_{\min}(n) = \Theta(n^2)$. 
\end{theorem} 

\begin{remark}[Interpretations and Implications of the Theorem]
Thus we are able to obtain 
asymptotic estimates for the time steps at which the vectors $av_0+bv_1 \in U_{\infty}$ 
(in the infinite Ulam set definition) enter the Ulam set by partitioning these points geometrically 
according to the $n^{th}$ line segments defined above to estimate the 
distribution of $K(a, b)$ on the $n^{th}$ line segment $L_n$, where we define 
$K(a, b)$ to be the value of the first step time $N$ where the vector $av_0+bv_1$ appears in the 
Ulam set $U_N$, but not in the set $U_{N-1}$. 
As it turns out, this is the most natural way to 
estimate the values in the timing plots in the figures 
without resorting to a much more complicated and chaotic two-parameter, 
three-dimensional plot timing construction for $K(a, b)$ where $a$ and $b$ are taken to be independent 
of one another. 

The timings on the $n^{th}$ line segments are a natural way to consider the entry times of 
individual vectors in the set since the $n^{th}$ line segments are approximately filled in completely one 
right after the other as the step times $N \geq 1$ increase, as we shall soon see in 
Proposition \ref{prop_PointEntryTimes_NthLineSegment} proved below. That is to say that the entry times of the 
vectors into the Ulam set are already naturally grouped into approximate time intervals by these line segments 
to begin with, so we simply estimate our timing statistics based on the points in these natural bins. 
A good heuristic explaining why we expect the asymptotic bound to be $\Theta(n^2)$ 
is that the number of points in the wedge on the line segments $L_k$ for $k=1,2,\ldots,n$ is quadratic in 
$n$, where we expect to add at most a bounded constant of $O(1)$ points at each distinct step 
(see Lemma \ref{lemma_UpperBoundOnNumPointsAdded_AtStepN} below) 
for a total of approximately a quadratic number of time steps needed to 
fill in $L_n$ completely. 
\end{remark} 

\begin{lemma}[Minimum and Maximum Magnitude Vectors] 
\label{lemma_MinMagVectors_OnLineSegment} 
Given all points on the $n^{th}$ line segment $L_n$, the minimum most magnitude 
vector on the line segment is in the boundary set $\{n v_0+v_1, v_0 + n v_1\}$ when 
$n$ is even, is given exactly by $\frac{n+1}{2}(v_0+v_1)$ when $n \equiv 1 \pmod{4}$, and is given exactly by 
$\frac{n+3}{2}v_0+\frac{n-1}{2}v_1 = \frac{n-1}{2}(v_0+v_1) + 2v_0$ when $n \equiv 3 \pmod{4}$. 
For all $n \geq 2$ the two vectors of the greatest (and not necessarily equal) magnitudes on the $n^{th}$ 
line segment are in the set $\{v_0 + n v_1, n v_0+v_1\}$. 
\end{lemma} 
\begin{proof} 
It is obvious that when $n$ is even, we only have two points on the $n^{th}$ segment $L_n$, so both claims 
are trivially true in this case. In the remaining two cases when $n$ is odd we can prove the claims by considering 
inequalities for large $n$. For odd $n$, we can write all vectors on the $n^{th}$ line segment in the form of 
$$w_{d} := \frac{(n+1-2d)}{2}(v_0+v_1) + 2dv_0,\ 
  \text{ for some $d \in \left[-\frac{n-1}{2}, \frac{n-1}{2}\right]$. }$$ 
The distinction between the two odd cases for $n$ modulo $4$ comes into play since 
when $n \equiv 1 \pmod{4}$ there are an odd number of vectors on the $n^{th}$ line segment, and when 
$n \equiv 3 \pmod{4}$ there are an even number of vectors on the $n^{th}$ line segment where we are 
essentially considering the middle-most vectors on the segment for our candidates for the minimum 
magnitude vectors on the segment. 

Let $v_0 := (x_0, y_0)$, $v_1 := (x_1, y_1)$, and let $\vartheta_{001}$ denote the angle between the vectors 
$v_0$ and $v_0+v_1$. 
Suppose that $n \equiv 1 \pmod{4}$. Then we have claimed that the minimum magnitude vector on the 
$n^{th}$ line segment is given by $w_0$. Since $x_0, x_1, y_0, y_1 \geq 0$ we can show 
that for large $n$ we have that $||w_0|| \leq ||w_1||$ and that $||w_0|| \leq ||w_{-1}||$. 
Now suppose that $d \geq 2$ and consider the differences of the magnitudes of the vectors $w_d$ and $w_{d+1}$: 
\begin{align*} 
||w_d||^2 & - ||w_{d+1}||^2 = \frac{(n+1-2d)^2}{4} ||v_0+v_1||^2 + 4d^2 ||v_0||^2 + 
     2d(n+1-2d) ||v_0|| \cdot ||v_0+v_1|| \cos(\vartheta_{001}) \\ 
     & \phantom{=\ } - 
     \left[ 
     \frac{(n-1-2d)^2}{4} ||v_0+v_1||^2 + 4(d+1)^2 ||v_0||^2 + 
     2(d+1)(n-1-2d) ||v_0|| \cdot ||v_0+v_1|| \cos(\vartheta_{001}) 
     \right],\ \\ 
     & \phantom{\qquad} \text{ by the polarization identity } \\ 
     & = 
     (n-2d) ||v_0+v_1||^2 - 4(2d+1) ||v_0||^2 - (2n-8d-2) ||v_0|| \cdot ||v_0+v_1|| \cos(\vartheta_{001}) \\ 
     & < 0,\ \text{ for $n$ large. } 
\end{align*} 
So when we increase the $d \geq 2$ the magnitude of the corresponding vector increases in magnitude. 
Similarly, for $d \geq 0$ we can easily show that the difference $||w_d|| - ||w_{d-1}|| < 0$. 
Thus we have proved our result in the first case of odd $n$. A similar modified argument shows that the 
second statement is also true. The inequalities we have established for the minimum magnitude vector cases 
show that the two maximum magnitude vectors on the $n^{th}$ line segment for $n$ odd are 
$\{v_0 + n v_1, n v_0+v_1\}$. 
\end{proof} 

\begin{prop}[Order of Point Entry Times on the $n^{th}$ Line Segment] 
\label{prop_PointEntryTimes_NthLineSegment}
For $n \geq 2$, the vectors on the $n^{th}$ line segment $L_n$ are filled in 
completely before the vectors in the $(n+1)^{th}$ line segment $L_{n+1}$ are filled in completely. 
The number of Ulam set points filled in on the $k^{th}$ segments $L_k$ for $k > n$ before the 
$n^{th}$ line segment is filled in completely is $O(n^2)$. 
\end{prop} 
\noindent 
\textbf{Restatement: } 
The proposition is restated in more formal notation as follows: 
For all $n \geq 2$ there exists a smallest time step $N \equiv N(n)$ such that 
$U_N \cap L_n = L_n$ but where $U_N \cap L_{n+k} \subsetneq L_{n+k}$ for all $k \geq 1$. 
Moreover, for each $n \geq 2$ and its minimal time step $N(n)$ we have that\footnote{ 
     This bound is likely significantly suboptimal, though it is sufficient to 
     complete the proof of the theorem as given below. 
} 
$$\sum_{k \geq 1} |U_{N(n)} \cap L_{n+k}| = O(n^2).$$
\begin{proof}
Let $n \geq 2$ be fixed. Since the maximum magnitude vector on the $n^{th}$ 
line segment is given by $v_0+nv_1$ (cf.\ Lemma \ref{lemma_MinMagVectors_OnLineSegment}), and 
we have that for all $k \geq 1$ the norms satisfy $||v_0+nv_1||_2 < ||v_0+(n+k)v_1||_2$, i.e., that 
\begin{equation*} 
||v_0+nv_1||_2 < ||v_0+(n+1)v_1||_2 < ||v_0+(n+2)v_1||_2 < \cdots, 
\end{equation*} 
we see that the first of the results stated above is true. 
By Lemma \ref{lemma_MinMagVectors_OnLineSegment}, to prove the 
second result we must bound the maximum $m$ of the 
minimums of each of the following sets by $O(n)$ where 
$m := \max \{m_{14}, m_{34}, m_{02}\}$ in the notation below: 
\begin{align*} 
m_{14} & := \min\ \left\{k \in \mathbb{Z}^{+}: \left\lVert \frac{(n+k+1)(v_0+v_1)}{2} \right\rVert - 
     ||v_0+nv_1||_2^2 > 0 \text{ and } n+k \equiv 1 \pmod{4}\right\} \\ 
m_{34} & := \min\ \left\{k \in \mathbb{Z}^{+}: \left\lVert \frac{(n+k-1)(v_0+v_1)}{2} + 2v_0\right\rVert - 
     ||v_0+nv_1||_2^2 > 0 \text{ and } n+k \equiv 3 \pmod{4}\right\} \\ 
m_{02} & := \min\ \left\{k \in \mathbb{Z}^{+}: ||(n+k)v_0+v_1||_2^2 - ||v_0+nv_1||_2^2 > 0 
     \text{ and } n+k \equiv 0 \pmod{2}\right\}. 
\end{align*} 
Then we have points in the later $k^{th}$ segments with $k > n$ for a maximum of $m$ incompletely 
filled line segments in the Ulam set where each of these segments is filled with $O(n)$ points. 
Thus to complete the proof we must show that $m = O(n)$. 
We illustrate our method for doing so by considering the third set case in the equations above. 
More precisely, if we let $v_0 := (x_0, y_0)$ and $v_1 := (x_1, y_1)$, 
we require that for some minimal $k \geq 1$ we have that 
\[
((n+k)x_0+x_1)^2 + ((n+k)x_0+x_1)^2 - \left[(x_0+nx_1)^2 + (y_0+ny_1)^2\right] > 0. 
\] 
When we reduce this inequality we see that the previous equation 
implies that $1 \leq k \leq O(n)$. The other two cases follow similarly by computation. 
Hence the total number of points in the subsequent line segments is bounded by 
$m \cdot O(n) = O(n^2)$.
\end{proof} 

\begin{lemma}[Upper Bound on the Number of Points Added at Step $N$] 
\label{lemma_UpperBoundOnNumPointsAdded_AtStepN}
For sufficiently large $N$, the 
maximum number of vectors added to the finite Ulam set at step $N$ is some bounded 
integer-valued constant, $C_{v_0,v_1}$, which depends only on the inner angle of the 
wedge formed by the fixed initial vectors $v_0, v_1$ in the first quadrant and the 
magnitudes of $v_0, v_1$, i.e., we have that 
$$\sup_{N \geq 1} \left(|U_N| - |U_{N-1}|\right) = C_{v_0,v_1} < \infty,$$
exists and is finite for each fixed initial vector pair $(v_0,v_1)$. 
\end{lemma} 
\begin{proof} 
Suppose that at some step $N$, $w_1$ is a minimal length vector to be added to the Ulam set 
with minimal $y$-coordinate. 
We need to count and bound the number of Ulam set points on the quarter circle in the 
first quadrant of radius $||w_1||_2$ whose $y$ components are strictly larger than that of $w_1$. 
We start the proof with exact estimates of the magnitudes of two proposed vectors on this quarter circle 
(assuming for the sake of argument that there are multiple minimum length vectors at step $N$) and then 
proceed to consider the approximate and asymptotic bounds on the parameters defining these vectors. 

If we suppose that our first minimal magnitude vector $w_N := (n+1-b)v_0+bv_1$ at time step $N$ is of 
minimal $y$-coordinate, then we may bound our constant at time step $N$ by the number of 
integer solutions $(t, s) \in \mathbb{Z}^2$ with $t = 0, 1, 2, \ldots, s$ to the equation 
\[
||(n+1-b)v_0+bv_1||_2 = ||(n+1+s - (b+t))v_0 + (b+t)v_1||_2, 
\] 
i.e., the number of integer solution pairs to the equation 
\[
s^2 ||v_0||^2 + t^2 ||v_1-v_0||^2 + 2st ||v_0|| \cdot ||v_1-v_0|| \cos(\vartheta_0-\vartheta_{01}) = 
     4 ||w_N||^2 \cos(\vartheta_w-|\vartheta_0-\vartheta_{01}|), 
\]
which is easier said than done \emph{exactly}. However, to count the number of solutions $s$ where 
$t \in [0, s]$, we note that we can alternately bound the number of times the quarter circle of radius 
$R := ||w_N||$ crosses a distinct line segment $L_{n+s}$, which leads to the upper bound of 
$4 \cdot \#(s)$ since there is a square root in the previous equation and since the quarter circle may 
intersect each $L_{n+s}$ it crosses at most twice. We then estimate $\#(s)$ by considering three points 
on the quarter circle: 1) $P_1$: the intersection with the line between $v_0$ and $v_0+v_1$, 
i.e., the line parameterized by $v_0+\lambda_0(v_0+v_1)$; 
2) $P_2$: the intersection with the line $v_1+\lambda_1(v_0+v_1)$; and 3) $P_3$: the point of longest distance from 
$v_0+v_1$ on the quarter circle, which we can intentionally overshoot by defining 
\[
P_3 := \frac{v_0+v_1}{||v_0+v_1||}\left(||v_0+v_1||+R\right). 
\]
Without loss of generality, we consider the distance between $P_1$ and $P_3$ and then divide through by the 
minimum of the magnitudes of $v_0,v_1$ (the approximate distance between the $L_n$) 
to obtain the desired upper bound for the number of line segments 
crossed by the quarter circle. In this case, we have that 
\[
\lambda_0 = \frac{\sqrt{R^2 - ||v_0||^2 - 2||v_0|| \cdot ||v_0+v_1|| \cos(\vartheta_0-\vartheta_{10})}}{||v_0+v_1||}, 
\]
where $\vartheta_{10}$ denotes the angle formed between the $x$-axis and the vector $v_0+v_1$. 
Then our upper bound on the constant $C_{v_0,v_1}$ is given by 
\[
C_{v_0,v_1} \leq \frac{8 \left\Vert \frac{v_0+v_1}{||v_0+v_1||}\left(||v_0+v_1||+R\right) - 
     \left(v_0+\lambda_0(v_0+v_1)\right) \right\rVert}{\min\ \{||v_0||, ||v_1||\}} 
     \xrightarrow{n \longrightarrow \infty}
     \frac{8}{\min\ \{||v_0||, ||v_1||\}}\left(||v_0+v_1|| + ||v_0||\right),  
\] 
where we have included an additional factor of $2$ in the upper bound due to possible symmetry in the initial vectors 
which can lead to two vectors on the same line segment $L_{n+s}$ having the same minimal magnitude. 
\end{proof} 

\begin{example}
To verify that our formula is in the ballpark of the expected largest constant, we plug-in our known 
prototypical vector case of $v_0,v_1 := (1, 0), (0, 1)$ to the formula above and maximize with respect to 
$R$. This leads us to the estimate that 
\begin{align*} 
C_{v_0,v_1} & \leq \frac{8}{\sqrt{2}} \left\lVert (1, 1) (R+\sqrt{2}) - \left(\sqrt{R^2-3}+\sqrt{2}, \sqrt{R^2-3}\right) 
     \right\rVert \\ 
     & = \frac{8}{\sqrt{2}} \sqrt{\left(R-\sqrt{R^2-3}\right)^2+\left(R-\sqrt{R^2-3}+\sqrt{2}\right)^2} \\ 
     & \leq 8 \sqrt{4+\sqrt{6}} \approx 20.3167, \text{ which occurs for $R := \sqrt{3}$. } 
\end{align*} 
Our empirical observations of this initial vector case place the constant at 
$C_{v_0,v_1} = 12$, so we see that we have obtained a 
realistic estimate for finite $N$ with this bound. 
If we require that $||v_0||,||v_1|| \in [1, Q]$ for some bounded real $Q > 1$, 
a requirement which we may impose by scaling, we can estimate that $C_{v_0,v_1} \leq 24 \cdot Q$ to 
obtain another upper bound on the number of vectors that can be added to an Ulam set at 
any prescribed time step $N$. 
\end{example} 

\begin{proof}[Proof of Theorem \ref{conj_UlamSetTimings_v1}] 
We begin by considering the maximum mean case where we determine exact 
bounds on the order of the function $T_{\max}(n)$ defined above over $n$. 
To reach the points in the $n^{th}$ line segment for $n \geq 2$ we have 
a finite number of vectors in the $k^{th}$ previous line segments to cover 
for $0 \leq k < n$. By Proposition \ref{prop_PointEntryTimes_NthLineSegment} 
we have that there are $2$ points on the $k^{th}$ line segment if 
$k$ is even and $(k+1)/2$ points on the $k^{th}$ line segment if $k$ is odd. 

\noindent 
\textbf{Upper Bound Estimate on $T_{\max}(n)$: } 
In the worst case analysis, 
we assume that we only add one vector to the finite Ulam set at each step $N$ and 
that the number of vectors filled in on the $k^{th}$ segments for $k > n$ when the 
last vector is filled in is bounded by $O(n^2)$ according to 
Proposition \ref{prop_PointEntryTimes_NthLineSegment}. 
Then to reach the last vector on the $n^{th}$ segment, by the proposition it 
takes $N$ steps where 
\begin{align*} 
N_n & = \sum_{k=0}^{n-1} \#_{\Ulam}(k) + 
     \begin{cases} 
     2 & \text{ if $n$ is even; } \\ 
     \frac{n+1}{2} & \text{ if $n$ is odd. } 
     \end{cases} + O(n^2) \\ 
     & = 
     \sum_{k=0}^{\left\lfloor \frac{n-1}{2} \right\rfloor} 2 + 
     \sum_{k=0}^{\left\lfloor \frac{n-1}{2} \right\rfloor} \frac{k+1}{2} + 
     \begin{cases} 
     2 & \text{ if $n$ is even; } \\ 
     \frac{n+1}{2} & \text{ if $n$ is odd. } 
     \end{cases} + O(n^2) \\ 
     & = 
     \frac{1}{4}\left(\left\lfloor \frac{n-1}{2} \right\rfloor+1\right) 
     \left(\left\lfloor \frac{n-1}{2} \right\rfloor+10\right) + 
     \begin{cases} 
     2 & \text{ if $n$ is even; } \\ 
     \frac{n+1}{2} & \text{ if $n$ is odd. } 
     \end{cases} + O(n^2) \\ 
     & = 
     O(n^2). 
\end{align*} 
\noindent 
\textbf{Lower Bound on $T_{\max}(n)$: } 
In the best case analysis, we have by 
Lemma \ref{lemma_UpperBoundOnNumPointsAdded_AtStepN} that we can add 
at most constant $C_{v_0,v_1}$ number independent of $N$ and $n$ 
of vectors to the growing finite Ulam set at each step $k$ for $k \geq N_0 \geq 1$, 
where we may assume for the sake of obtaining a lower bound that the $n^{th}$ 
line segment is filled in 
completely before any points on the $k^{th}$ line segments for $k > n$. 
Thus to reach the last vector on the $n^{th}$ segment in this case we 
require $N$ steps where 
\begin{align*} 
N_n & \geq \sum_{k=N_0}^{n-1} O\left(\frac{\#_{\Ulam}(k)}{C_{v_0,v_1}}\right) + 
     \begin{cases} 
     O\left(\frac{2}{C_{v_0,v_1}}\right) & \text{ if $n$ is even; } \\ 
     O\left(\frac{n+1}{2 C_{v_0,v_1}}\right) & \text{ if $n$ is odd. } 
     \end{cases} = O(n^2). 
\end{align*} 
Hence $T_{\max}(n) = \Theta(n^2)$. The computation for the minimum mean case is 
similar. 
\end{proof} 

\section{Conclusions} 

We hope that this collection of data on tiling statistics and suggestions of theoretical methods to compute limiting distributions inspires further study in this area. We believe new methods will be needed to study many of the tilings, particularly those without obvious symmetries (or hidden symmetries in the case of quasicrystals). 
The tilings project website provides forms for interactive side-by-side 
comparisons of our plot data beyond the similarities between related 
tilings we suggest in this article. 
Our data is publicly available and freely usable, however, 
we do request that any use of the data cites this work and the website where the data is hosted.

\section*{Acknowledgments} 

The author would like to thank Jayadev S. Athreya at the University of Washington in Seattle for 
suggesting the problems in the article, for funding to work on the tilings software project, and for 
his guidance throughout the project. 

\renewcommand{\refname}{References}

\setcounter{section}{1}
\renewcommand{\thesection}{\Alph{section}}

\begin{sidewaysfigure} 
\centering
\begin{minipage}[b]{0.95\linewidth}
\includegraphics[scale=0.75,frame]{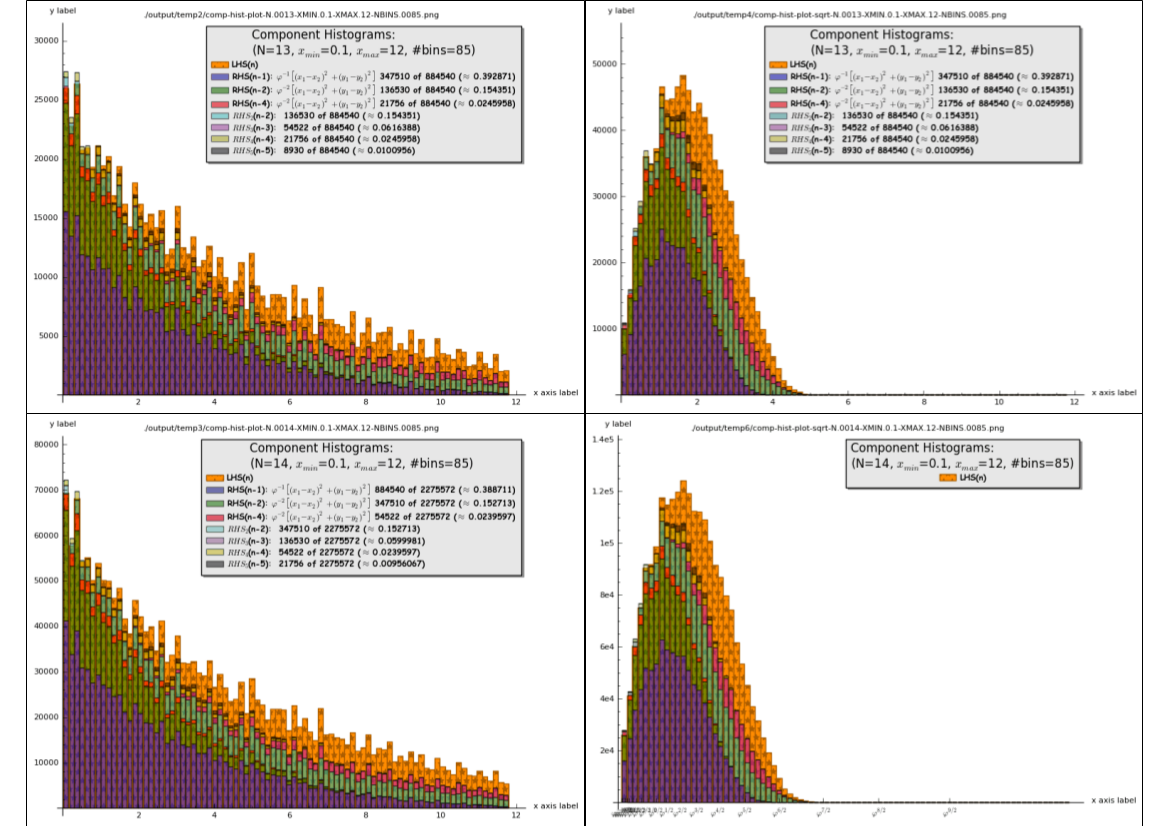} 
\end{minipage} 

\caption{Stacked Pair Correlation Plots for $N : 13$ and $N := 14$. 
Key: 
\textit{Orange} = Full Distribution (for visual comparision), 
\textit{Blue} = $\RHS(n-1) := \varphi^{-1} \widetilde{D}_{n-1}$, 
\textit{Green} = $\RHS(n-2) := \varphi^{-2} \widetilde{D}_{n-2}$, 
\textit{Red} = $\RHS(n-4) := \varphi^{-4} \widetilde{D}_{n-4}$, 
\textit{Turquoise} = $\RHS_2(n-2)$, 
\textit{Pink} = $\RHS_3(n-3)$, 
\textit{Yellow} = $\RHS_4(n-4)$, and 
\textit{Gray} = $\RHS_5(n-5)$. 
} 
\label{figure_AmmannChair_StackedPCPlots}
\end{sidewaysfigure} 

\begin{figure}[ht]
\centering 

\begin{minipage}{0.85\linewidth} 
\centering\includegraphics[clip,height=0.27\textheight,width=\linewidth,keepaspectratio,frame]{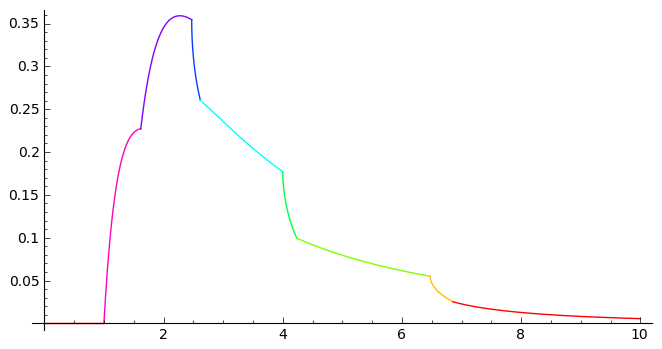}
\caption*{The piecewise smooth empirical distribution for the 
          slope gaps of saddle connections on the 
          golden L (vectors of slope at most $1$ and horizontal component 
          less that $10^{4}$).} 
\end{minipage} 

\bigskip\bigskip\bigskip

\begin{minipage}{0.85\linewidth} 
\centering\includegraphics[clip,height=0.27\textheight,width=\linewidth,keepaspectratio,frame]{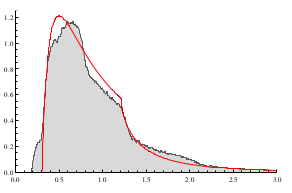}
\caption*{Empirical distribution of the slope gaps of the points in a long 
          patch of the \texttt{TubingenTriangle} tiling 
          (approximately $1.5 \times 10^6$ points). 
          The red comparision curve corresponds to 
          \emph{Hall's distribution} for the slope gap distribution of the 
          two-dimensional integer lattice.} 
\end{minipage} 

\bigskip\bigskip

\caption{Empirical slope gap distributions for the \texttt{SaddleConnGoldenL} 
         and \texttt{TubingenTriangle} tilings (permission to reproduce the 
         second image from \cite{TUBINGEN-GAPS} given by T. Jakobi)} 
\label{figure_EmpSlopeGapDist_SConnGoldenL_TT_tilings} 

\end{figure} 

\begin{sidewaysfigure} 
\centering
\begin{minipage}[b]{0.95\linewidth}
\includegraphics[scale=0.75,frame]{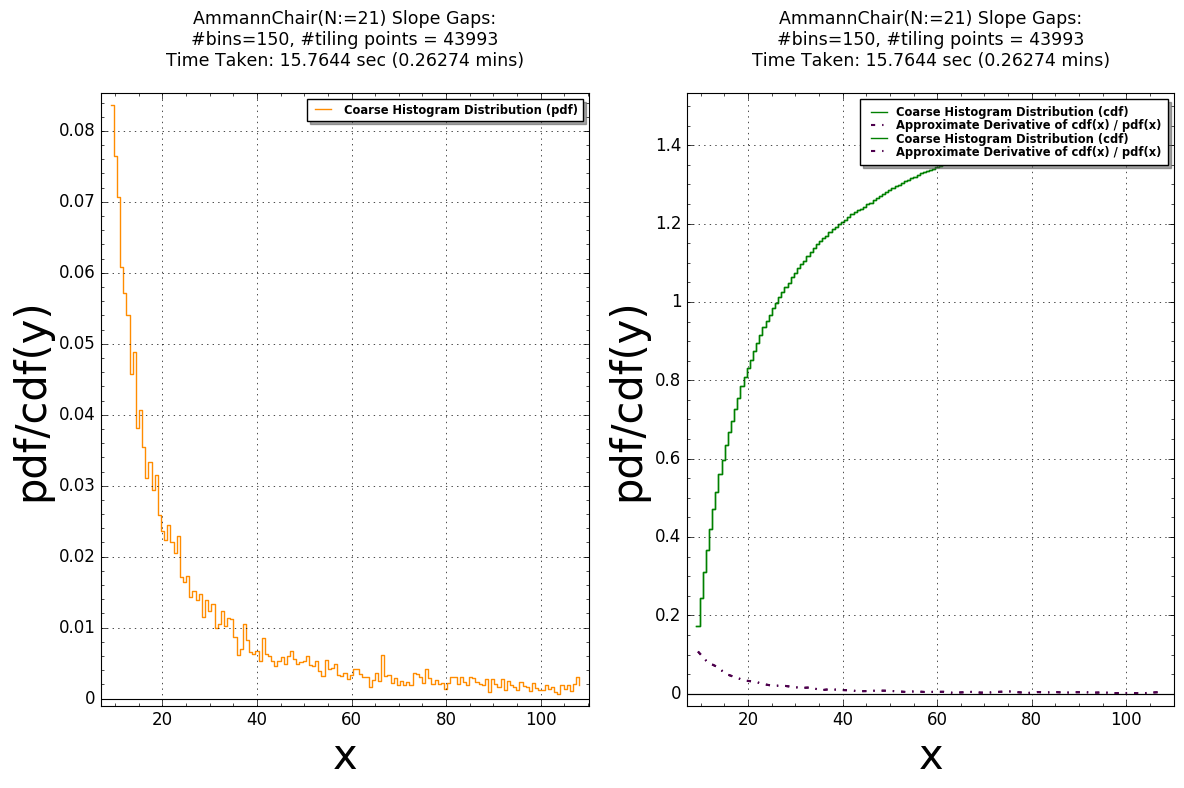} 
\end{minipage} 

\caption{Empirical slope gap distribution of the \texttt{AmmannChair} tiling 
         after $N := 21$ substitution steps} 
\label{figure_AmmannChair_EmpiricalSlopeGapDist}
\end{sidewaysfigure} 

%

%

\begin{figure}[h] 

\bigskip\hrule\bigskip 

\centering 
\begin{minipage}{\linewidth} 

\begin{minipage}{0.5\linewidth} 
\centering\includegraphics[trim={0 0 85 60},clip,height=2.2in,width=\linewidth,keepaspectratio,frame]{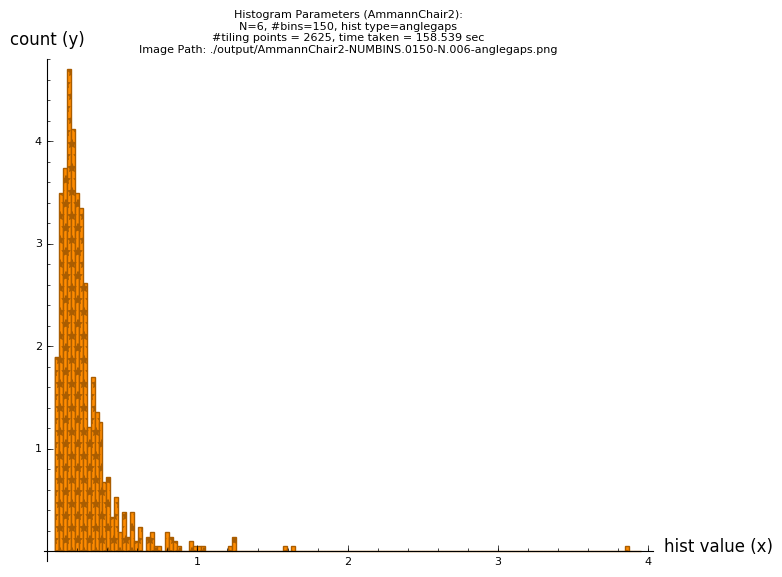}
\caption*{\texttt{AmmannChair2} Angle Gaps} 
\end{minipage} 
\begin{minipage}{0.5\linewidth} 
\centering\includegraphics[trim={0 0 85 60},clip,height=2.2in,width=\linewidth,keepaspectratio,frame]{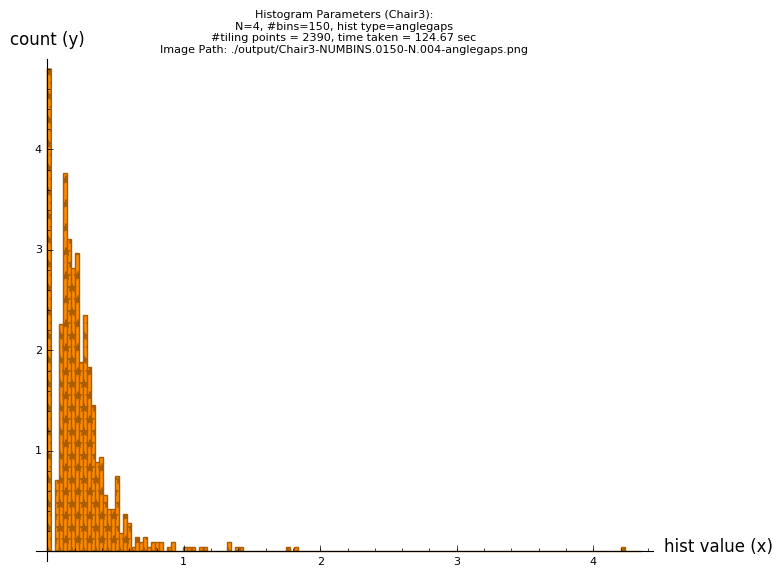}
\caption*{\texttt{Chair3} Angle Gaps} 
\end{minipage} 

\smallskip 

\end{minipage} 

\begin{minipage}{\linewidth} 

\begin{minipage}{0.5\linewidth} 
\centering\includegraphics[trim={0 0 85 60},clip,height=2.2in,width=\linewidth,keepaspectratio,frame]{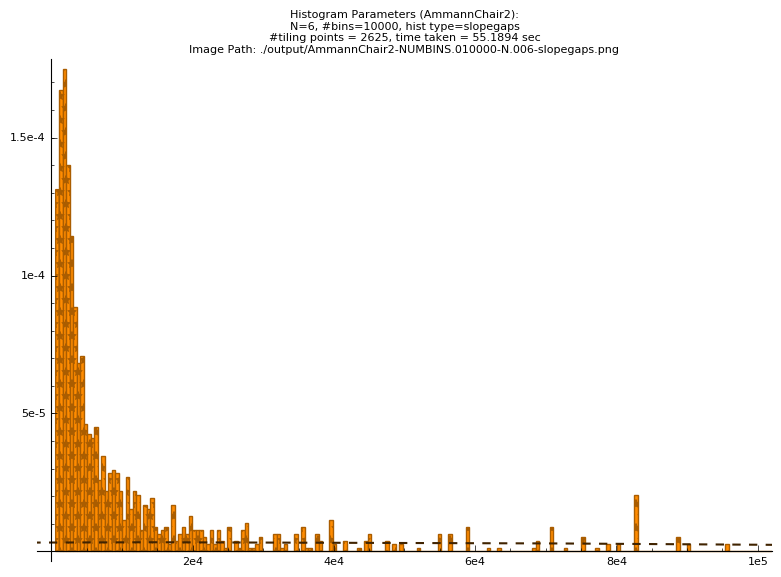}
\caption*{\texttt{AmmannChair2} Slope Gaps} 
\end{minipage} 
\begin{minipage}{0.5\linewidth} 
\centering\includegraphics[trim={0 0 85 60},clip,height=2.2in,width=\linewidth,keepaspectratio,frame]{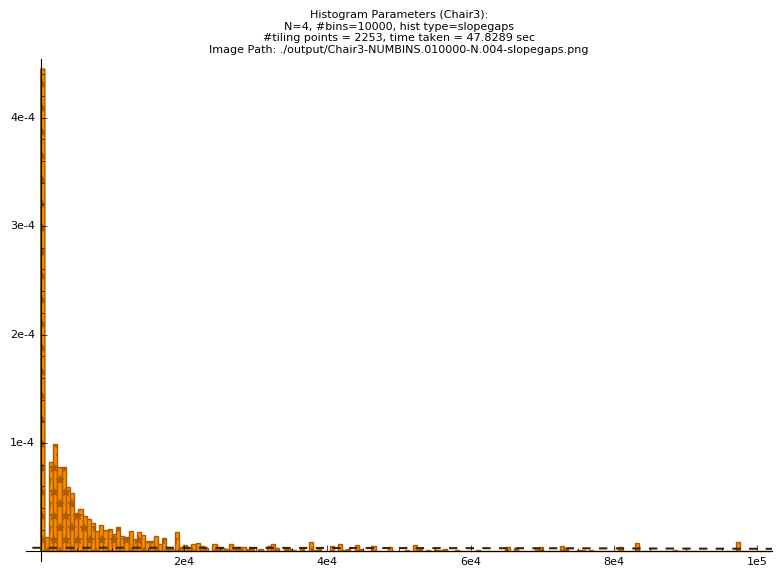}
\caption*{\texttt{Chair3} Slope Gaps} 
\end{minipage} 

\smallskip 

\begin{minipage}{\linewidth} 

\begin{minipage}{0.5\linewidth} 
\centering\includegraphics[trim={0 0 85 60},clip,height=2.2in,width=\linewidth,keepaspectratio,frame]{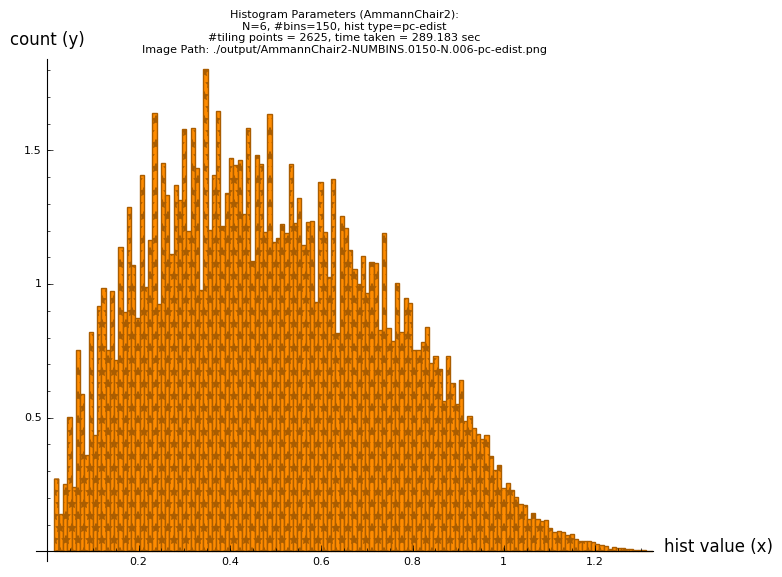}
\caption*{\texttt{AmmannChair2} Pair Correlation} 
\end{minipage} 
\begin{minipage}{0.5\linewidth} 
\centering\includegraphics[trim={0 0 85 60},clip,height=2.2in,width=\linewidth,keepaspectratio,frame]{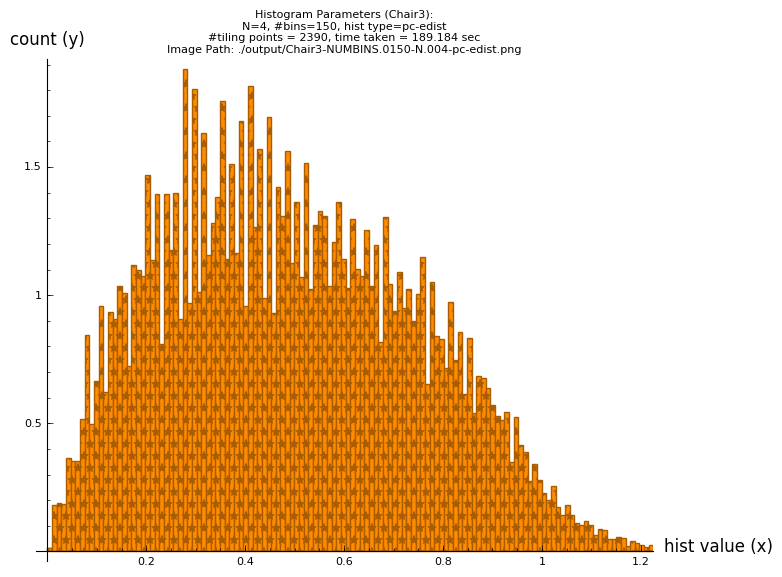}
\caption*{\texttt{Chair3} Pair Correlation} 
\end{minipage} 

\end{minipage} 

\end{minipage}

\caption{A Comparision of the Empirical Distributions of the 
         \texttt{AmmannChair2} ($N := 6$) and the 
         \texttt{Chair3} ($N := 4$) Tilings} 
\label{figure_AmmannChair2_Chair3_Comps} 
\label{figure_ChairTilineVariants_DistsSummary}

\bigskip\hrule\bigskip 

\end{figure}

\begin{figure}[ht] 

\bigskip\hrule\bigskip 

\centering 
\begin{minipage}{\linewidth} 

\begin{minipage}{0.5\linewidth} 
\centering\includegraphics[trim={0 0 85 60},clip,height=2.2in,width=\linewidth,keepaspectratio,frame]{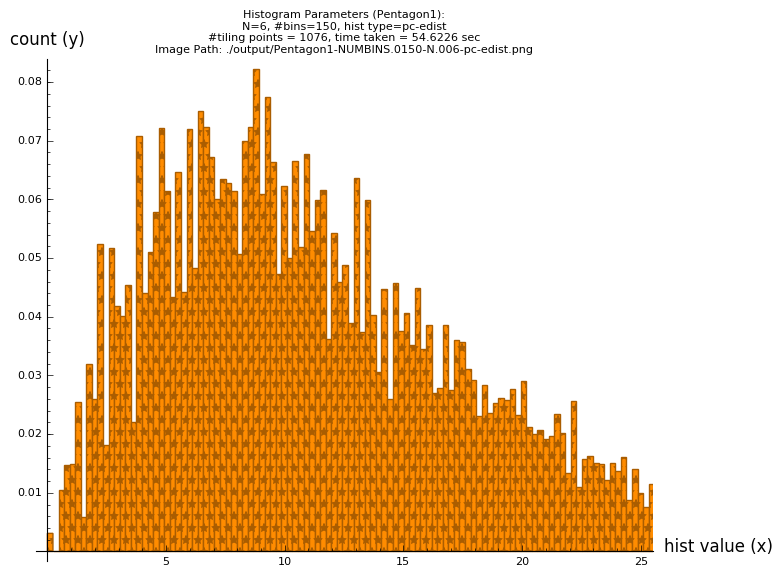}
\caption*{\texttt{Pentagon1} Pair Correlation} 
\end{minipage} 
\begin{minipage}{0.5\linewidth} 
\centering\includegraphics[trim={0 0 85 60},clip,height=2.2in,width=\linewidth,keepaspectratio,frame]{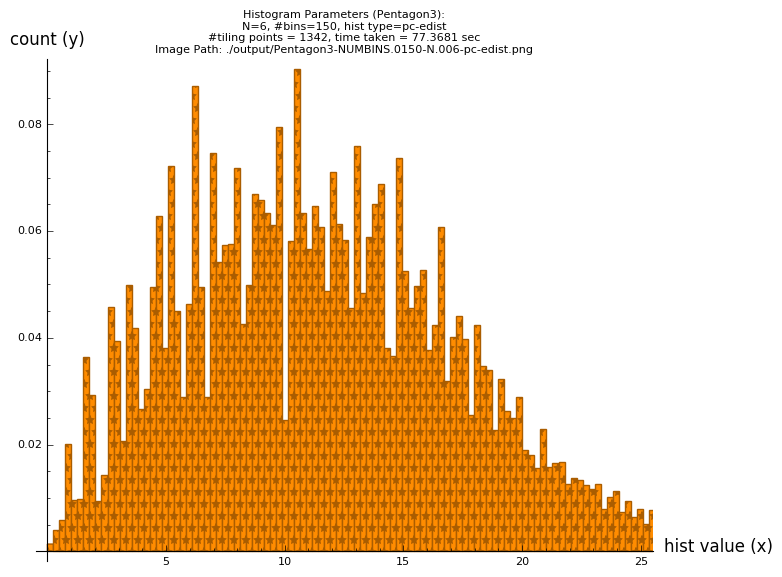}
\caption*{\texttt{Pentagon3} Pair Correlation} 
\end{minipage} 

\smallskip 

\end{minipage} 

\begin{minipage}{\linewidth} 

\begin{minipage}{0.5\linewidth} 
\centering\includegraphics[trim={0 0 85 60},clip,height=2.2in,width=\linewidth,keepaspectratio,frame]{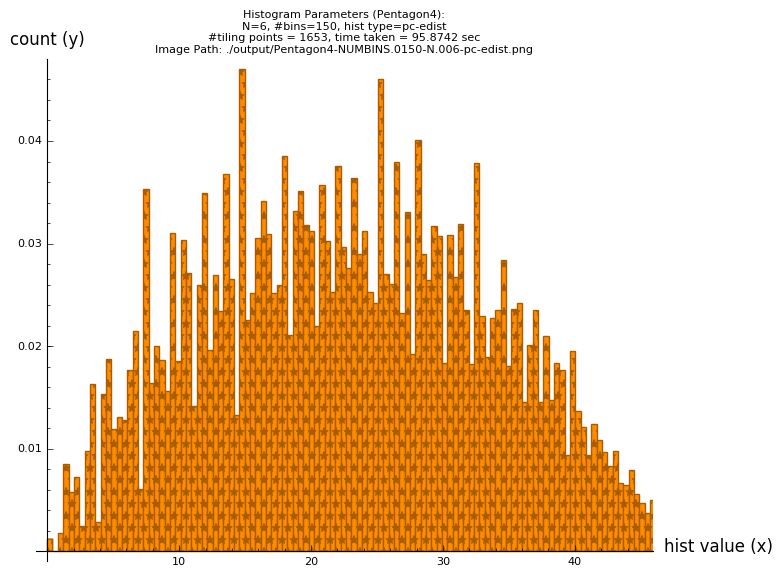}
\caption*{\texttt{Pentagon4} Pair Correlation} 
\end{minipage} 
\begin{minipage}{0.5\linewidth} 
\centering\includegraphics[trim={0 0 85 60},clip,height=2.2in,width=\linewidth,keepaspectratio,frame]{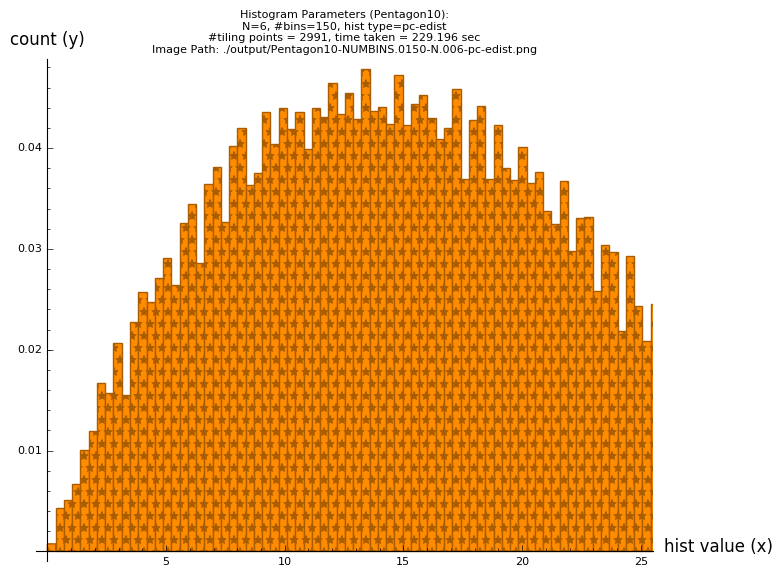}
\caption*{\texttt{Pentagon10} Pair Correlation} 
\end{minipage} 

\smallskip 

\begin{minipage}{\linewidth} 

\begin{minipage}{0.5\linewidth} 
\centering\includegraphics[trim={0 0 85 60},clip,height=2.2in,width=\linewidth,keepaspectratio,frame]{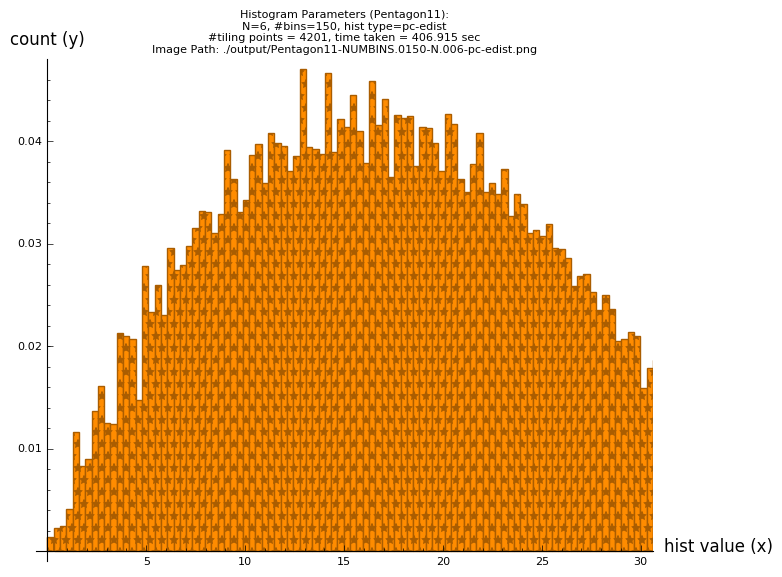}
\caption*{\texttt{Pentagon11} pc-edist} 
\end{minipage} 
\begin{minipage}{0.5\linewidth} 
\centering\includegraphics[trim={0 0 85 60},clip,height=2.2in,width=\linewidth,keepaspectratio,frame]{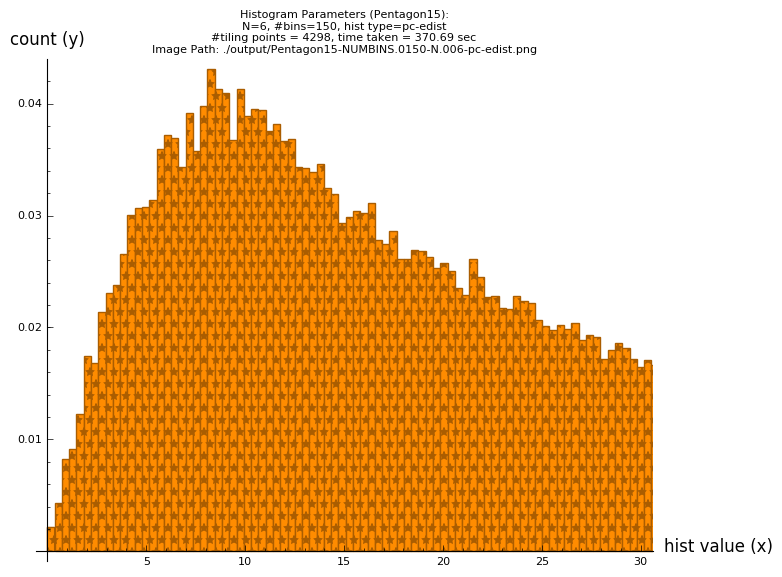}
\caption*{\texttt{Pentagon15} Pair Correlation} 
\end{minipage} 

\end{minipage} 

\end{minipage} 

\caption{A Comparison of the Pair Correlation Plots for a Subset of the Pentagon Tilings} 
\label{figure_PCDists_PentagonTilings_Comp}

\bigskip\hrule\bigskip 

\end{figure} 

\begin{figure}[ht]  

\centering  
\begin{minipage}{0.95\linewidth} 
\centering\includegraphics[height=0.44\textheight,keepaspectratio,frame]{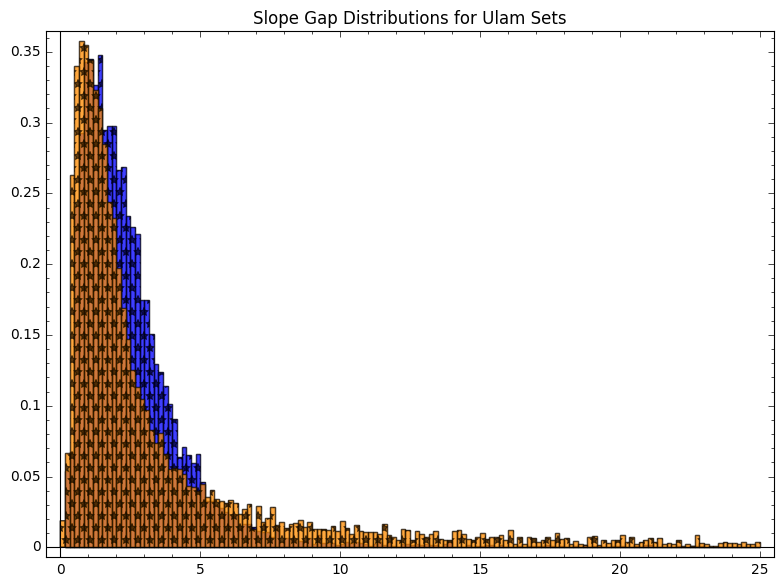}
\end{minipage} 

\centering  
\begin{minipage}{0.95\linewidth} 
\centering\includegraphics[height=0.44\textheight,keepaspectratio,frame]{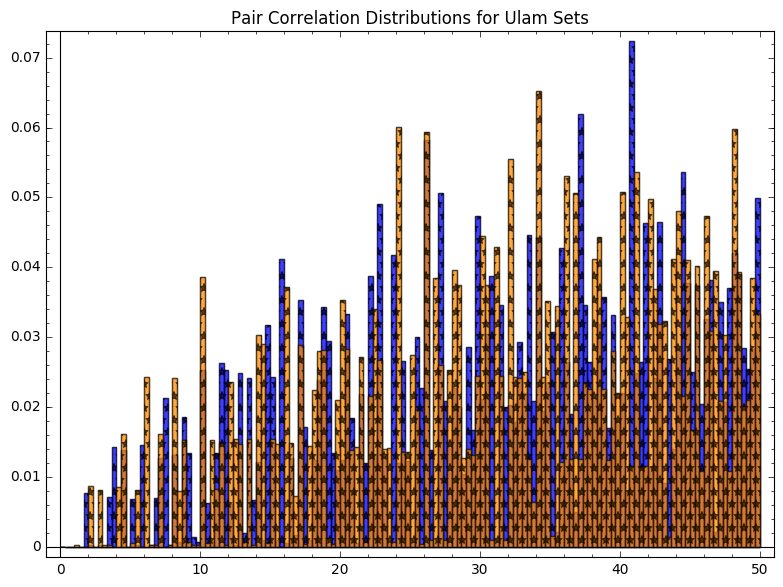}
\end{minipage}

\caption{A Comparison of the slope gap and pair correlation 
         empirical distributions where the slope gaps are scaled by a factor of the wedge angle 
         between $v_0, v_1$ for $v_0,v_1 = (0, 1), (1, 0)$ (\emph{orange}) and 
         $v_0,v_1 = (1, \varphi), (\varphi, 1)$ (\emph{blue}) when $N:= 5000$} 
\label{figure_proto_UlamSetDists_v1}

\end{figure}

\end{document}